\documentclass[10pt]{amsart}
\pdfoutput=1

\usepackage[english]{babel}
\usepackage{amssymb}
\usepackage{amsxtra}
\usepackage{hyperref}
\usepackage{tikz}
\usepackage{booktabs}
\usepackage{multirow}
\usepackage{mathtools}
\usetikzlibrary{snakes}

\theoremstyle{plain}
\newtheorem{theorem}{Theorem}[section]
\newtheorem{proposition}[theorem]{Proposition}
\newtheorem{lemma}[theorem]{Lemma}
\theoremstyle{definition}
\newtheorem{remark}[theorem]{Remark}
\newtheorem{definition}[theorem]{Definition}

\DeclareMathOperator{\Aut}{Aut}
\DeclareMathOperator{\ch}{\chi}
\DeclareMathOperator{\diag}{diag}
\DeclareMathOperator{\Ext}{Ext}
\DeclareMathOperator{\ext}{ext}
\DeclareMathOperator{\ho}{h}
\DeclareMathOperator{\Hom}{Hom}
\DeclareMathOperator{\gengeom}{p_g}
\DeclareMathOperator{\GL}{GL}
\DeclareMathOperator{\I}{I}
\DeclareMathOperator{\id}{id}
\DeclareMathOperator{\irr}{q}
\DeclareMathOperator{\Mor}{Mor}
\DeclareMathOperator{\N}{N}
\DeclareMathOperator{\Obj}{Obj}
\DeclareMathOperator{\Pic}{Pic}
\DeclareMathOperator{\Spec}{Spec}
\newcommand{\Ho}{\mathord{\mathrm H}}
\newcommand{\T}{\mathord{\mathrm T}}
\DeclareMathOperator{\Tors}{Tors}

\DeclareMathOperator{\sO}{\mathcal{O}}
\DeclareMathOperator{\sK}{\mathcal{K}}
\DeclareMathOperator{\bA}{\mathbb{A}}
\DeclareMathOperator{\bC}{\mathbb{C}}
\DeclareMathOperator{\bP}{\mathbb{P}}
\DeclareMathOperator{\bZ}{\mathbb{Z}}

\newcommand{\code}[1]{\texttt{#1}}
\newcommand{\name}[1]{\textit{#1\/}}
\newcommand{\obar}[1]{\bar{#1}}
\let\oldsqcup\sqcup{}
\renewcommand{\sqcup}{\mathbin{\,\oldsqcup\,}}
\let\oldbigsqcup\bigsqcup{}
\renewcommand{\bigsqcup}{\mathop{\,\oldbigsqcup\,}}

\newcommand{\affmod}{\tilde{M}} 
\newcommand{\schmod}{M} 
\newcommand{\stamod}{\mathcal{M}} 
\newcommand{\trumod}{\mathfrak{M}} 

\title[Automorphisms of numerical Godeaux surfaces\dots]{Automorphisms of numerical Godeaux surfaces with torsion of order 3, 4, or 5}
\author{Stefano Maggiolo}

\begin{document}
  \begin{abstract}
    We compute the automorphisms groups of all numerical Godeaux surfaces, i.e.\ minimal smooth surfaces of general type with $\sK^2 = 1$ and $\gengeom = 0$, with torsion of the Picard group of order $\nu$ equals $3$, $4$, or $5$. We present explicit stratifications of the moduli spaces whose strata correspond to different automorphisms groups.
    
    Using the automorphisms computation, for each value of $\nu$ we define a quotient stack, and prove that for $\nu = 5$ this is indeed the moduli stack of numerical Godeaux surfaces with torsion of order $5$. Finally, we describe the inertia stacks of the three quotient stacks.
  \end{abstract}
  
  \maketitle
  \setcounter{tocdepth}{1} 
  \tableofcontents
  
  \section{Introduction}
    \emph{Numerical Godeaux surfaces\/} are the algebraic surfaces of general type with the smallest invariants, $\sK^2 = 1$ and $\gengeom = 0$. For this reason they have been studied thoroughly in the history of the classification of algebraic surfaces. Conjectured to be rational by Max Noether as a subclass of the surfaces with $\gengeom = 0$ and $\irr = 0$, they take their name from Lucien Godeaux: in 1931, he constructed one of them, providing the first example of minimal surface of general type with $\gengeom = 0$. This particular example is called \emph{Godeaux surface}.
    
    A first classification appears in~\cite{M1976} by Miyaoka: numerical Godeaux surfaces are split in five classes up to their torsion group. In~\cite{R1978}, Reid constructs the moduli spaces of the three classes with larger torsion group. Up to now, even if several examples of surfaces in the other two classes are known, there are no similar constructions for them.

    Recently, another viewpoint has been pursued: the observation that all sporadical examples of numerical Godeaux surfaces with small torsion group admit an involution led to the study of numerical Godeaux surfaces with an involution. This study has been completed by Calabri, Ciliberto and Mendes Lopes in~\cite{CCM2007}, who prove a classification theorem for such surfaces. The following step, classification of numerical Godeaux surfaces with an automorphism of order three, has been completed by Palmieri in~\cite{P2008}, who found that there are no such surfaces.
    
    Halfway between the two viewpoints, we consider the problem of finding all the automorphisms of the surfaces for which the moduli spaces is known, i.e.\ the ones with large torsion group. Using the constructions found by Reid, we are able to compute explicitly the automorphisms groups of such surfaces. The results are then split into strata of the moduli spaces and gathered in the three tables~\ref{tab:z5},~\ref{tab:z4}, and~\ref{tab:z3}. An interesting observation is that the surfaces with the lower torsion amongst the one we consider are all rigid, that is, they do not admit any nontrivial automorphism.
    
    We observe that the way in which the automorphisms are computed reminds of the way one construct a quotient stack. Indeed, we prove that the moduli stack of numerical Godeaux surfaces with torsion group of order five is a quotient stack. We do this using the automorphisms computation, so we are able to describe explicitly the structure of this stack.
    
    The paper is organized as follows. In \S~\ref{sec:godeaux}, we recall the construction of the moduli spaces of numerical Godeaux surfaces with large torsion group given by Reid. In \S~\ref{sec:computing} we present the strategy used to compute the automorphisms groups, giving a more detailed description for the easier case as an example; also, we describe briefly the computer program we wrote to do most part of the work. We collect the results obtained in \S~\ref{sec:results}. In \S~\ref{sec:stacks} we define the moduli stacks of numerical Godeaux surfaces and prove that the one of surfaces with torsion of order five is a quotient stacks. In \S~\ref{sec:inertia} we compute the inertia stacks of the quotient stacks coming from the automorphisms computation.

    \paragraph{\textbf{Acknowledgements}}
    The computation of the automorphisms groups of numerical Godeaux surfaces of order $4$ or $5$ are part of my Master Thesis at University of Pisa. I would like to heartily thank Rita Pardini, who introduced me to the problem and followed me carefully during and after the development of the Thesis.
    
    I would like to acknowledge my current advisor, Barbara Fantechi, for suggesting to me the topic of the second part of this work and helping me to develop it.
    
    I am also very grateful to Miles Reid for helpful discussions.

    \paragraph{\textbf{Notations}}
    A \emph{surface\/} will be a (smooth) projective algebraic variety of dimension $2$, defined over $\bC$. If $S$ is a surface, we will denote with $\sK_S$ its canonical divisor; the \emph{geometric genus\/} of $S$ is $\gengeom(S) = \ho^0(S, \sK_S) = \ho^2(S, \sO_S)$ and its \emph{irregularity\/} is $\irr(S) = \ho^1(S, \sK_S) = \ho^1(S, \sO_S)$. The \emph{characteristic\/} of $S$ is $\ch(S) = 1 - \irr(S) + \gengeom(S)$. The \emph{torsion group\/} $\Tors(S)$ is the torsion subgroup of $\Pic(S)$. The permutation matrix associated to the permutation $\sigma$ is denoted as $P_\sigma$.
  
  \section{Numerical Godeaux surfaces}\label{sec:godeaux}
    In this section we will recall briefly what a numerical Godeaux surface is and how to construct the coarse moduli space for numerical Godeaux surfaces with torsion isomorphic to $\bZ_5$, $\bZ_4$ and $\bZ_3$. We will follow the article~\cite{R1978} with some insights from the newer work~\cite{R2000}. Let us start with the raw definition.
    
    \begin{definition}
      A \emph{numerical Godeaux surface\/} is a minimal smooth surface of general type $S$ with $\sK_S^2 = 1$, $\gengeom(S) = \irr(S) = 0$ so that $\ch(S) = 1$. For brevity, in the following we will write simply \emph{Godeaux surfaces\/} for them.
    \end{definition}
    
    Already in~\cite{M1976} Godeaux surfaces are classified by their torsion group, which is a group of order less or equal to $5$. In~\cite{R1978}, it is also proved that it cannot be $\bZ_2 \times \bZ_2$, therefore there are only five possibilities: the cyclic groups of order from $5$ down to $2$, and the trivial group. We restrict our attention to the three classes with larger torsion groups, so that our main tool will be the study of the universal Galois cover (that is constructed via the torsion group), and the relations between the canonical ring of a Godeaux surface and of its cover. To fix notations, from now on $S$ will be a Godeaux surface and $\psi\colon X \to S$ the cover associated to its torsion group $G$.
    
    \subsection{Torsion of order five}\label{subsection:five}
      This is a basic computation, since we can easily find the invariants of $X$ and check that in particular $\sK_X^2 = 2 \gengeom(X) - 3$, i.e.\ $X$ is a Horikawa surface. Then from~\cite{H1976} we know that $X$ is birational to a quintic hypersurface in $\bP = \bP(x_1, x_2, x_3, x_4)$, with $\deg x_i = 1$, by the canonical map $\phi\colon X \to \obar{X} \subset \bP$. Moreover, $\obar{X}$ has at most rational double points as singularities.
    
      The group $G$ acts naturally on $X$ and on $\Ho^0(\sK_X)$; so $\Ho^0(\sK_X)$ is a $G$-module and we know that this $G$-module decomposes as the direct sum of the four nontrivial characters of $G$ (see for example Proposition 2.4 in~\cite{MP2008}). Therefore we may assume that the action $\rho$ of $G$ on $\bP$ is fixed and generated by the automorphism $\diag(\xi, \xi^2, \xi^3, \xi^4)$, where $\xi$ is a fixed primitive fifth root of unity. Moreover, we will often identify $G$ with its image in $\Aut(\bP)$.
    
      Hence we have to classify all quintic hypersurfaces $\obar{X} \subset \bP$, fixed by this action, with at most rational double points. We will not specify explicitly the locus of quintics that do not satisfy the latter condition; as for the former, we only have to require that the monomials composing the equation of $\obar{X}$ are in the same eigenspace of $\Ho^0(X, 5 \sK_X)$ with respect to the action of $G$. Since $\obar{X}$ cannot pass through the fixed points of the action (which are the coordinate points), in the equation there are necessarily the monomials $x_i^5$, so the eigenspace is fixed to be the one containing these monomials. Summing up, the equation is of this kind:
      \begin{equation}\label{eq:z5}
        \begin{split}
          q_0 &= x_1^5 + x_2^5 + x_3^5 + x_4^5 +\\
          &\qquad + b_1 x_2 x_3^3 x_4 + b_2 x_1^3 x_3 x_4 + b_3 x_1 x_2 x_4^3 + b_4 x_1 x_2^3 x_3 +\\
          &\qquad + c_1 x_2^2 x_3 x_4^2 + c_2 x_1 x_3^2 x_4^2 + c_3 x_1^2 x_2^2 x_4 + c_4 x_1^2 x_2 x_3^2\text{.}
        \end{split}
      \end{equation}

      We have eight affine parameters; to obtain the coarse moduli space, we have to remove the points which give surfaces with singularities worse than rational double points and to quotient by isomorphisms of the correspondent Godeaux surfaces. Such an isomorphism lifts as an isomorphism of $\bP$ which sends the first $\obar{X}$ in the second and commutes with $\rho$. As we will see in Remark~\ref{remark:finite_group}, for any surface there are only finitely many points corresponding to surfaces isomorphic to the first one; this means that we are quotienting by a finite group (this is true also for the next two cases). Its action is far from being free, nevertheless we have the following.
    
      \begin{theorem}
        The coarse moduli space $\schmod_5$ of Godeaux surface with torsion of order $5$ is a finite quotient of a nonempty open subset $\affmod_5$ of $\bA^8$. A point \[(b_1, b_2, b_3, b_4, c_1, c_2, c_3, c_4)\] corresponds to the Godeaux surface obtained resolving the singularities of the quotient by $\rho$ of the variety defined by equation~\eqref{eq:z5} in $\bP$.
      \end{theorem}
    
    \subsection{Torsion or order four}
      For the next two cases, we will compute generators and relations of the canonical ring of $X$. We may use direct computation using the property of the canonical ring of $S$, or gain help from the numerator of the Hilbert series as noted in~\cite{R2000}. For torsion of order $4$, we need five generators: $x_1$, $x_2$, $x_3$ in degree $1$, and $y_1$, $y_3$ in degree $2$ (the subscripts denote the eigenspace in which the generators lie). So $X$ is naturally embedded in the weighted projective space $\bP = \bP(1^3, 2^2)$. As for the relations, we have two of them in degree $4$, $q_0$ and $q_2$ (again, the subscripts denote the eigenspaces). These generators and relations describe the canonical ring of $X$, and one proves that the bicanonical map $\phi\colon X \to \obar{X} \subset \bP$ is a birational morphism and $\obar{X}$ has at most rational double points. Again, we can fix the action $\rho$ of $G$ (so that it is diagonal) on $\bP$ and exploit its fixed locus to eliminate some parameters from $q_0$ and $q_2$. After some simplifications, they assume these forms:
      \begin{equation}\label{eq:z4}
        \begin{split}
          q_0 &= x_1^4 + x_2^4 + x_3^4 + a x_1^2 x_3^2 + a^\prime x_1 x_2^2 x_3 +\\
          &\qquad + y_1 y_3 + b_1 y_1 x_1 x_2 + b_3 y_3 x_2 x_3 \text{,}\\
          q_2 &= c_1 x_1^3 x_3 + c_3 x_1 x_3^3 +\\
          &\qquad + d_1 x_1^2 x_2^2 + d_3 x_2^2 x_3^2 + y_1^2 + y_3^2\text{.}
        \end{split}
      \end{equation}

      As in the previous case, we have eight parameters, and we have to eliminate the points which give bad singularities and to quotient by the isomorphisms of underlying Godeaux surfaces.
      
      \begin{theorem}
        The coarse moduli space $\schmod_4$ of Godeaux surfaces with torsion of order $4$ is a finite quotient of a nonempty open subset $\affmod_4$ of $\bA^8$. A point \[(a, a^\prime, b_1, b_3, c_1, c_3, d_1, d_3)\] corresponds to the Godeaux surface obtained resolving the singularities of the quotient by $\rho$ of the variety defined by equations~\eqref{eq:z4} in $\bP$.
      \end{theorem}

      \subsection{Torsion of order three}
        Using the same methods as before, we need six generators for the canonical ring of $X$: $x_1$, $x_2$ in degree $1$, $y_0$, $y_1$, $y_2$ in degree $2$, $z_1$, $z_2$ in degree $3$. Therefore we should use the tricanonical map to obtain the canonical model of $X$; actually, we can use just the bicanonical map, since one proves that it is a birational morphism to $\bP = \bP(1^2, 2^3)$. The image of this morphism is not a complete intersection; it is described by equations:
      \begin{equation}\label{eq:z3}
        \begin{split}
          q_0 &= x_1 x_2 (y_0^2 - y_1 y_2) - x_1^2 (y_2^2 - y_0 y_1) - x_2^2 (y_1^2 - y_0 y_2) +\\
          &\qquad + a_1 x_1^3 x_2 y_1 + a_2 x_1 x_2^3 y_2 - b_1 x_1^6 + b_{1,2} x_1^3 x_2^3 - b_2 x_2^6 \text{,}\\
          p_0 &= y_0^3 + y_1^3 + y_2^3 - 3 y_0 y_1 y_2 + a_1 x_1^2 y_0 y_1 + a_2 x_2^2 y_0 y_2 -\\
          &\qquad - (a_1+a_2) x_1 x_2 y_1 y_2 + a_1 x_2^2 y_1^2 + a_2 x_1^2 y_2^2 + (b_1+b_{1,2}+b_2) x_1^2 x_2^2 y_0 + \\
          &\qquad + b_2 x_2^4 y_1 - (b_1+b_{1,2}) x_1^3 x_2 y_1 + b_1 x_1^4 y_2 - (b_{1,2}+b_2) x_1 x_2^3 y_2 + \\
          &\qquad + (x_1^3 + x_2^3) S\text{,}\\
          h &= x_1 y_1 (y_2^2 - y_0 y_1) + x_2 y_2 (y_1^2 - y_0 y_2) - a_1 x_1^2 x_2 y_1^2 - a_2 x_1 x_2^2 y_2^2 -\\
          &\qquad (b_1 x_1^3 + b_2 x_2^3) x_1 x_2 y_0 + b_1 x_1^5 y_1 + b_2 x_2^5 y_2 - x_1^2 x_2^2 S\text{,}
        \end{split}
      \end{equation}
      where $S = c_1 x_1^3 + c_2 x_2^3 + d_1 x_1 y_2 + d_2 x_2 y_1$.
      
      Actually, omitting $h$ we have the surface $\obar{X}$ plus three fibers of the projective bundle $\bP \to \bP^1$ (obtained projecting to the first two coordinates), restricted to $(q_0 = 0)$. Moreover, the parameters are not uniquely determined, since they may change by a transformation of the form $x_i \mapsto k x_i$, $y_i \mapsto y_i$, $z_i \mapsto k^{-1} z_i$. Accounting for these transformations, we have the following.
        
      \begin{theorem}
        The coarse moduli space $\schmod_3$ of Godeaux surfaces with torsion of order $3$ is a finite quotient of a nonempty open subset $\affmod_3$ of the weighted projective space $\bP(2^2,4^3,6^2,4^2)$. A point \[[a_1, a_2, b_1, b_{1,2}, b_2, c_1, c_2, d_1, d_2]\] corresponds to the Godeaux surface obtained resolving the singularities of the quotient by $\rho$ of the variety defined by equations~\eqref{eq:z3} in $\bP$.
      \end{theorem}

  \section{Computing automorphisms}\label{sec:computing}
    This section is the technical heart of the paper: in the first subsection we will discuss the mathematics needed to solve the problem; in the second, as an example, we will apply it to the case of torsion of order $\nu = 5$, without doing any hard computation; in the third we will explain the structure of the program that does the computations, referring to the example for clarifications.
    
    Here we let $S$ be again a Godeaux surface with torsion isomorphic to $\bZ_\nu$ with $\nu \geq 3$, and $\psi\colon X \to S$ its universal Galois cover. Moreover we take $\phi\colon X \to \obar{X} \subset \bP$ to be the canonical (in the case of $\nu = 5$) or bicanonical (in the other cases) birational morphism as constructed in the previous section.
  
    \subsection{The strategy}
      We will use without further mention the following facts.
      \begin{enumerate}
        \item An automorphism of $X \subset \bP$ extends to an automorphism of $\bP$ (in particular is described by a matrix in $\bP \GL(n+1)$).
        \item An isomorphism of two Godeaux surfaces $S_1$ and $S_2$ lifts to an automorphism of the universal covers $X_1$ and $X_2$.
        \item The automorphisms of a universal cover $X$ that pass to the quotient to automorphisms of $S$ are the ones compatible with the action of $G$, i.e.\ the ones in the normalizer of $\Aut(X)$ relative to the action of $G$. The kernel of the map $\N_{\Aut(X)}(G) \to \Aut(S) \to 0$ is simply $G$.
      \end{enumerate}
    
      In every case we studied before, we fixed the action of the torsion group $G$ on the projective space $\bP$; hence the compatibility with the action does not depend on the particular equations of $\obar{X}$. Up to now we can describe $\Aut(S)$ as the quotient by $G$ of $\Aut(X)$, and we can represent elements of $\Aut(X)$ by matrices in $\bP \GL(n+1)$. Firstly, we have to reduce the possibilities for these matrices.
      
      \begin{lemma}
        A matrix $A$ representing an automorphism $\alpha$ of $X$ is a permutation of a diagonal matrix. In particular, the permutation is induced by the action of $\alpha$ on the eigenspaces of $\Ho^0(n \sK_X)$ relative to $G$.
      \end{lemma}
    
      \begin{proof}
        The automorphism $\alpha$ induces an automorphism $\alpha^\star$ of $\Ho^0(X, n \sK_X)$ for every $n \geq 0$. This gives as a result that $A$ is a block matrix (i.e.\ $\alpha^\star(\bullet_i) = \bullet_j$ for some $j$, where $\bullet$ represents the letter $x$ or $y$, when it makes sense). Since $\alpha$ is by definition compatible with the action of $G$, also $\alpha^\star$ is compatible with the action of $G$ on $\Ho^0(X, n \sK_X)$, so there is an induced permutation of the eigenspaces relative to $G$, i.e.\ $\alpha^\star$ acts on the characters of $G$ as an element of $\Aut(G) \cong G^\star$.

      It is now easy to see that $A$ has to be a permutation of a diagonal matrix; indeed, if $\alpha$ is in the class relative to $g \in \bZ_\nu^\star$, then $\alpha^\star$ sends $\bullet_i$ to $\bullet_{g i}$ and this is well defined since in our cases for every $n$ and every $g$ the eigenspace of $\Ho^0(X, n \sK_X)$ with eigenvalue $g$ is at most $1$-dimensional.\qedhere
      \end{proof}
      
      So, if $\nu = 5$ the possible automorphisms are divided in four classes ($1$, $2$, $3$, and $4$); if $\nu = 4$ there are two classes ($1$ and $3$); if $\nu = 3$ there are two classes ($1$ and $2$). Moreover, for every class we have a fixed permutation and these permutations form a group isomorphic to $\bZ_\nu^\star$. Hence, the automorphisms groups we will find will be semidirect products of some finite group by a subgroup of $\bZ_\nu^\star$.

      Let us denote with $\gamma_\nu$ the number of generators of the canonical ring considered in the previous section, or equivalently the dimension of the projective space in which $\obar{X}$ is embedded. Up to now, to describe an automorphism of $X$ we need $\gamma_\nu-1$ complex parameters (because we work in $\Aut(\bP)$). In the following Lemma, we show that these parameters cannot be generic.
      
      \begin{lemma}\label{lemma:parameters are roots}
        Up to normalizing, the nonzero entries of $A$ are $\nu$-th roots of unity.
      \end{lemma}
      
      \begin{proof}
        We define $\lambda_i$ and $\mu_i$ implicitly by $\alpha^\star x_i = \lambda_i x_{g i}$ and $\alpha^\star y_i = \mu_i y_{g i}$.
    
        \begin{description}
        \item[Case $\nu = 5$] in the equation $q_0$ defining $X$ we have the terms $x_i^5$; hence, to send $q_0$ in a multiple of itself, the ratios of the four parameters must be fifth roots of unity; if we normalize one of them to $1$, the others must be of the form $\xi^{i_g}$.
        \item[Case $\nu = 4$] if we normalize $\lambda_2$ to $1$, we have from $q_0$ that $\lambda_1$, $\lambda_3$, and $\mu_1 \mu_3$ are fourth roots of unity; from $q_2$ we have that $\mu_1^2 = \mu_3^2$, hence also $\mu_1$ and $\mu_3$ are fourth roots of unity. Moreover we also observe that $\mu_3 = \mu_1^{-1}$.
        \item[Case $\nu = 3$] we can normalize $\lambda_1$ to $1$, so that from $p_0$ we have $\mu_1 = \mu \xi^{j_1}$, $\mu_2 = \mu \xi^{j_2}$ and $\mu_0 = \mu \xi^{2 j_1 + 2 j_2}$, and from $q_0$ we have \[ \lambda_2 \mu^2 \xi^{j_1 + j_2} = \mu^2 \xi^{2 j_2} = \lambda_2^2 \mu \xi^{2 j_1}\text{;}\] from these equations, we get $\lambda_2 = \xi^{2 j_1 + j_2}$. We still have a continuous parameter; to kill it, we have to exploit the fact that $a_1 + a_2 \neq 0$ (see from~\cite{R1978}): this allows us to say that $\mu_0^3 = \lambda_2 \mu_1 \mu_2$, i.e.\ $\mu^3 = \xi^{2 j_2}$.\qedhere
        \end{description}
      \end{proof}
      
      Let $P_\sigma$ be the matrix associated to the permutation $\sigma$, corresponding to the multiplication by an element of $\bZ_\nu^\star$. Then Lemma~\ref{lemma:parameters are roots} tells us that the possible matrices $A$ representing an automorphism $\alpha$ of $X$ are of the following forms.
      \begin{description}
        \item[Case $\nu = 5$] $A = \diag(1, \xi^{i_2}, \xi^{i_3}, \xi^{i_4}) P_\sigma$.
        \item[Case $\nu = 4$] $A = \diag(\xi^{i_1}, 1, \xi^{i_3}, \xi^{j_1}, \xi^{-j_1}) P_\sigma$.
        \item[Case $\nu = 3$] taking $k = 2 j_1 + j_2$, $A = \diag(1, \xi^k, \xi^k, \xi^{2k}, 1) P_\sigma$.
      \end{description}
      
      \begin{remark}\label{remark:finite_group}
        Up to now, we did not use that $\alpha(X) = X$; that is, these matrices represent all the isomorphisms between points of $\affmod_\nu$. Let $H_\nu \subset \Aut(\bP)$ be the group consisting of all these matrices; then $H_\nu = \N_{\Aut(\bP)}(G)$ and $\affmod_\nu / H_\nu$ is the coarse moduli space $\schmod_\nu$ of Godeaux surfaces with torsion of order $\nu$. In particular, it is $\bZ_5^3 \ltimes \bZ_5^\star$ for $\nu = 5$, $\bZ_4^3 \ltimes \bZ_4^\star$ for $\nu = 4$ and the symmetric group of order $6$, $\bZ_3 \ltimes \bZ_3^\star$, for $\nu = 3$. Note that the torsion group $G \subset H_\nu$ acts trivially on $\affmod_\nu$.
      \end{remark}
    
      Coming back to automorphisms, we have proved that for a given $X$, there is only a finite group of possible automorphisms. Depending on the actual equations defining $X$, $\Aut(X)$ is a subgroup of that group. In particular, we experience changing of $\Aut(X)$ when some parameters becomes zero or when the ratios of two parameters related by a permutation becomes a $nu$-th root of unity. If we work on the parameters' affine space instead that on the coarse moduli space, i.e.\ on $\affmod_\nu = \bA^8$ for $\nu \in \{5, 4\}$, and on $\affmod_\nu = \bA^9$ for $\nu = 3$, these changes of $\Aut(X)$ happens only in some vector subspaces of $\affmod_\nu$. Even with this simplification, there are eight or nine parameters which gives hundreds of different cases; this is the reason to use a program to automate the computations.
      
    \subsection{Example: the case $\nu = 5$}  
      To explain the program, we present the easier case. When $\nu = 5$, we have four generators, $x_1$, $x_2$, $x_3$ and $x_4$, with degree $1$, with only one relation in degree $5$. An automorphism $\alpha$ of $X$ is represented by a matrix of the form $\diag(1, \xi^{i_2}, \xi^{i_2}, \xi^{i_3}) P_\sigma$, where $\sigma$ is the permutation given by the multiplication by an element of $\bZ_5^\star$. We have to compute the possible automorphisms, one permutation a time.
      
      If $\sigma$ is the identity, then the generic equation~\eqref{eq:z5} is transformed by $\alpha$ to
      \begin{align*}
        \alpha^\star q_0 &= x_1^5 + x_2^5 + x_3^5 + x_4^5 +\\
        &\qquad + b_1 \xi^{i_2+3i_3+i_4} x_2 x_3^3 x_4 + b_2 \xi^{i_3+i_4} x_1^3 x_3 x_4 +\\
        &\qquad\qquad + b_3 \xi^{i_2+3i_4} x_1 x_2 x_4^3 + b_4 \xi^{3i_2+i_3} x_1 x_2^3 x_3 +\\
        &\qquad +c_1 \xi^{2i_2+i_3+2i_4} x_2^2 x_3 x_4^2 + c_2 \xi^{2i_3+2i_4} x_1 x_3^2 x_4^2 +\\
        &\qquad\qquad + c_3 \xi^{2i_2+i_4} x_1^2 x_2^2 x_4 + c_4 \xi^{i_2+2i_3} x_1^2 x_2 x_3^2\text{.}
      \end{align*}
      Since the terms $x_s^5$ are unchanged, the condition on $i_2$, $i_3$ and $i_4$ for $\alpha$ to fix $X$ is $\alpha^\star q_0 = q_0$, i.e.\ this system of equations in $\bZ_5$:
      \begin{equation*}
        \left\{
          \begin{array}{l}
            i_2+3i_3+i_4 \equiv 0,\\
            i_3+i_4 \equiv 0,\\
            i_2+3i_4 \equiv 0,\\
            3i_2+i_3 \equiv 0,
          \end{array}\quad
          \begin{array}{l}
            2i_2+i_3+2i_4 \equiv 0,\\
            2i_3+2i_4 \equiv 0,\\
            2i_2+i_4 \equiv 0,\\
            i_2+2i_3 \equiv 0\text{,}
          \end{array}
        \right. \iff
        \left\{
          \begin{array}{l}
            i_3 \equiv 2i_2,\\
            i_4 \equiv 3i_2\text{.}
          \end{array}
        \right.
      \end{equation*}
      Obviously this is so if all $b_s$ and $c_s$ are nonzero; if some of them are zero, then the associated equations are not in the system.
      
      If $\sigma$ is not the identity, there are also some swaps amongst the coefficients $b_s$ and $c_s$; for example if $\sigma$ corresponds to the multiplication by $4$, we have
      \begin{align*}
        \alpha^\star q_0 &= x_1^5 + x_2^5 + x_3^5 + x_4^5 +\\
        &\qquad + b_4 \xi^{3i_2+i_3} x_2 x_3^3 x_4 + b_3 \xi^{i_2+3i_4} x_1^3 x_3 x_4 +\\
        &\qquad\qquad + b_2 \xi^{i_3+i_4} x_1 x_2 x_4^3 + b_1 \xi^{i_2+3i_3+i_4} x_1 x_2^3 x_3 +\\
        &\qquad + c_4 \xi^{i_2+2i_3} x_2^2 x_3 x_4^2 + c_3 \xi^{2i_2+i_4} x_1 x_3^2 x_4^2 +\\
        &\qquad\qquad + c_2 \xi^{2i_3+2i_4} x_1^2 x_2^2 x_4 + c_1 \xi^{2i_2+i_3+2i_4} x_1^2 x_2 x_3^2\text{,}
      \end{align*}
      and for $\alpha$ to fix $X$ we need again $\alpha^\star q_0 = q_0$. But because of the nontrivial permutation, necessary conditions to have an automorphisms are that $b_1 / b_4$ is a fifth root of unity and the same for all the coefficients swapped. We define $n_{s,t}$ and $m_{s,t}$ in such a way that $b_s / b_t = \xi^{n_{s,t}}$ and $c_s / c_t = \xi^{m_{s,t}}$, assuming this is possible. This time the system of equations will give conditions not only on the entries of $\alpha$, but also on the coefficients of $q_0$:
      \begin{multline*}
        \left\{
          \begin{array}{l}
            n_{1,4} \equiv 3i_2+i_3,\\
            -n_{3,2} \equiv i_2+3i_4,\\
            n_{3,2} \equiv i_3+i_4,\\
            -n_{1,4} \equiv i_2+3i_3+i_4,
          \end{array}\quad
          \begin{array}{l}
            m_{1,4} \equiv i_2+2i_3,\\
            -m_{3,2} \equiv 2i_2+i_4,\\
            m_{3,2} \equiv 2i_3+2i_4,\\
            -m_{1,4} \equiv 2i_2+i_3+2i_4\text{,}
          \end{array}
        \right. \iff\\
        \iff \left\{
          \begin{array}{l}
            n_{3,2} \equiv 2n_{1,4},\\
            m_{1,4} \equiv 2n_{1,4},\\
            m_{3,2} \equiv 4n_{1,4}
          \end{array}
        \right. \quad \wedge \quad \left\{
          \begin{array}{l}
            i_3 \equiv 2i_2+n_{1,4},\\
            i_4 \equiv 3i_2+n_{1,4}
          \end{array}
        \right.
      \end{multline*}
      This means that even if all coefficients are nonzero, we will have automorphisms with this permutation only in the five four-dimensional subspaces of $\affmod_5$ defined by
      \begin{align*}
        b_1 &= b_4 \xi^{n_{1,4}}\text{,} & b_3 &= b_2 \xi^{2 n_{1,4}}\text{,}\\
        c_1 &= c_4 \xi^{2 n_{1,4}}\text{,} & c_3 &= c_2 \xi^{4 n_{1,4}}\text{.}
      \end{align*}
      
      Continuing with the last two permutations, doing all the computations, and combining all the data collected, we arrive at the complete description of the automorphisms group of Godeaux surfaces with torsion of order $5$. 
      
    \subsection{Description of the program}
      In this section we will describe the program we wrote to compute the automorphisms groups. It is available at~\cite{M2009}. It is written in \name{Python}, using the library \name{sympy} to handle symbolic computation. It also uses \name{GAP}, mainly to identify the groups we obtain at the end of the computation.
      
      Here are the main classes, with their methods.
      \begin{enumerate}
        \item The class \code{GAPInterface} connects the main program with \name{GAP}. Its public methods are:
          \begin{itemize}
            \item \code{NullSpaceMat}, which returns the kernel of a matrix given as input;
            \item \code{IdSmallGroup}, which returns the id of a group in the \name{GAP}'s small group list; the group is passed as a list of generators and a list of relations.
          \end{itemize}
          It uses internally \code{rewrite\_expr}, which translate an expression from \name{sympy} to \name{GAP}.
        \item The class \code{LinearModularParametricSystem} solves a linear system in the ring $\bZ_\nu$, $3 \leq \nu \leq 5$; it is parametric in the sense that some unknowns are treated as parameters, and, in the solution, the the value of a parameter cannot depend on a regular unknown. Its main methods are:
          \begin{itemize}
            \item \code{solve}, with the obvious meaning;
            \item \code{iter\_solutions}, which returns an iterator through all the possible values of the regular unknowns (eventually depending on the parameters);
            \item \code{gens\_sample\_solutions}, same as before, but substituting a sample values for the parameters (i.e.\ all zeroes) and returning only the generators of the solutions;
            \item \code{iter\_pars\_solutions}, which returns an iterator through all the possible values of the parameters.
          \end{itemize}
        \item The class \code{VectorSpace} implements complex vector subspace: it takes as input two lists, of generators and of linear equations.
        \item The class \code{GodeauxAutomorphismComputer} is where the actual computation is done. We will describe it in detail later.
      \end{enumerate}
      
      We have three functions which define the needed data for the three cases and call \code{GodeauxAutomorphismComputer}. The input data are the following (between parenthesis the data for the example $\nu = 5$):
      \begin{enumerate}
        \item \code{n} ($5$), the order of the torsion group;
        \item \code{monomials} ($x_1^5, \dots, x_1^2 x_2 x_3^2$), the monomials involved in the equations of $X$;
        \item \code{mod\_pars} ($b_1, \dots, b_4, c_1, \dots c_4$), the basis of $\affmod_\nu$;
        \item \code{cr\_gens} ($x_1, \dots, x_4$), the generator of $\Ho^0(X, \sK_X)$ or $\Ho^0(X, 2\sK_X)$; in the latter case, if the generators were $x_s$ in degree $1$ and $y_s$ in degree $2$, we put the $y_s$ and the products $x_s x_t$ denoted as $x_{s,t}$;
        \item \code{cr\_rels} (equation~\eqref{eq:z5}), the relations between elements of \code{cr\_gens}, depending on \code{mod\_pars}, excluding the trivial ones such as $x_{s,t} x_{u,v} = x_{s,u} x_{t,v}$;
        \item \code{cr\_rels\_multiplier} ($1$), the coefficients of the relations \code{cr\_rels}, after applying an automorphism; since in every equation there is a constant monomial, we know this coefficient;
        \item \code{sys\_unks} ($i_2, i_3, i_4$), the list of unknown exponents of $\xi$ in the definition of the general automorphism $\alpha$;
        \item \code{sys\_pars} ($n_{i,j}, m_{i,j}$), the list of possible parameters showing up in the computations;
        \item \code{sys\_pars\_coupling} ($(b_i, b_j) \mapsto n_{i,j}, \dots$), a dictionary that associate a parameter in \code{sys\_pars} to a ratio between two coefficients in \code{mod\_pars};
        \item \code{alpha} ($\diag(1, \xi^{i_2}, \xi^{i_3}, \xi^{i_4})$), the generic automorphism with $\sigma = 1$;
        \item \code{perms} ($I, P = P_{(2,4,3,1)}, P^2, P^3$), a dictionary that associate to a number in $\bZ_\nu^\star$ the permutation matrix; in particular we get the generic automorphism with permutation $h$ as \code{alpha * perms[h]};
        \item \code{rho} ($\diag(1, \xi, \xi^2, \xi^3)$), the matrix representing a generator for the action of $\bZ_\nu$ on $\bP$;
        \item \code{psi} ($\diag(\xi, \xi, \xi, \xi)$), generators for the group to quotient by to obtain $\bP\GL(n+1)$ from $\GL(n+1)$; this is needed since \name{GAP} does not understand projective matrices groups.
      \end{enumerate}
      
      The class \code{GodeauxAutomorphismComputer} splits the computation in three steps, each of which consisting in a private method.
      \begin{enumerate}
        \item The first method, \code{compute\_equations}, builds the dictionary \code{equations}, indexed by permutations and pairs of elements of \code{mod\_pars}, of modular equations that will compose the systems to be solved. For example, if $\nu = 5$, the entry corresponding to the permutation $4$ and parameters $(b_4, b_1)$ is the equation $-n_{1,4} \equiv i_2+3i_3+i_4$. The dictionary is built applying the generic automorphism \code{alpha * perms[h]} and comparing the coefficients of the elements of \code{monomials}.
        \item The second method, \code{compute\_solutions}, iterates through all possible vanishing of elements in \code{mod\_pars}, that is in ${\{0,1\}}^8$ or ${\{0,1\}}^9$; for every vanishing and every permutation, it takes the equations from \code{equation} and call \code{LinearModularParametricSystem} to solve it. After this, it computes the relations between the parameters needed to have solutions, that is, the vector subspace where the solutions live. It builds a dictionary, \code{automorphisms\_gens}, indexed by the various vector spaces and with values the set of matrices generating the automorphisms group found solving the system. The last thing it does is to propagate the set of automorphisms of a larger vector space $V$ to the set of vector space contained in $V$.
        \item The third method, \code{regroup\_solutions}, takes all these informations, spread in all the vector spaces and collects them together. Firstly, it computes \name{GAP}'s id for all the possible set of generators, and build a dictionary indexed by these ids and with values the list of vector spaces which have that group as automorphisms group. Then it remove from these lists irrelevant vector spaces, that is the ones that are contained in a different space with the same automorphisms group.
      \end{enumerate}
      
  \section{The results}\label{sec:results}
    \subsection{Torsion of order five}
      The results given by the program are listed in Table~\ref{tab:z5} (where $s, t \in \bZ_5$ and $u, v \in \bZ_5^\star$).
      
      \begin{table}[h]
        \begin{tabular}{lccllcc}
          \toprule
          Group & \name{GAP} id & V. sp. & \multicolumn{2}{c}{Equations} & Dim. & Comp.\\
          \midrule[\heavyrulewidth]
          $\{1\}$ & $(1, 1)$ & $\affmod_5$ & & & $8$ & $1$\\
          \midrule
          \multirow{2}{*}{$\bZ_2$} & \multirow{2}{*}{$(2, 1)$} & \multirow{2}{*}{$Q_s$} & $b_1 = b_4 \xi^s$ & $c_1 = c_4 \xi^{2s}$ & \multirow{2}{*}{$4$} & \multirow{2}{*}{$5$}\\
          & & & $b_3 = b_2 \xi^{2s}$ & $c_3 = c_2 \xi^{4s}$ & &\\
          \midrule
          \multirow{3}{*}{$\bZ_4$} & \multirow{3}{*}{$(4, 1)$} & \multirow{3}{*}{$P_{s,t}$} & $b_1 = b_2 \xi^{s+t}$ & $c_1 = c_2 \xi^{2s+2t}$ & \multirow{3}{*}{$2$} & \multirow{3}{*}{$25$}\\
          & & & $b_2 = b_4 \xi^{s}$ & $c_2 = c_4 \xi^{2s}$ & &\\
          & & & $b_3 = b_1 \xi^{3s+t}$ & $c_3 = c_1 \xi^{s+2t}$ & &\\
          \midrule
          $\bZ_5$ & $(5, 1)$ & $H_u$ & \multicolumn{2}{l}{$b_v = c_v = 0$, $\forall v \neq u$} & $2$ & $4$\\
          \midrule
          $\bZ_5^2 \ltimes \bZ_4$ & $(100, 10)$ & $O$ & \multicolumn{2}{l}{$b_v = c_v = 0$, $\forall v$} & $0$ & $1$\\
          \bottomrule
        \end{tabular}
        \caption{Special subcomponents in the case $\nu = 5$.}\label{tab:z5}
      \end{table}

      We also have the relations of containment amongst the various vector spaces, recorded in Figure~\ref{fig:z5} (a vertical path means that the space at the lower end is contained in the one at the upper end).

      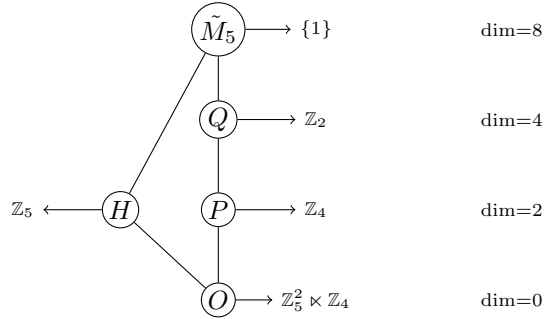
\begin{figure}[h]
        \begin{tikzpicture}
        \def\x{1.3}
        \def\y{-1.2}
        \node[shape=circle,draw,inner sep=0.03cm] (A1_3) at (1*\x, 3*\y) {$O$};
        \node[shape=circle,draw,inner sep=0.03cm] (A0_2) at (0*\x, 2*\y) {$H$};
        \node[shape=circle,draw,inner sep=0.03cm] (A1_2) at (1*\x, 2*\y) {$P$};
        \node[shape=circle,draw,inner sep=0.03cm] (A1_1) at (1*\x, 1*\y) {$Q$};
        \node[shape=circle,draw,inner sep=0.03cm] (A1_0) at (1*\x, 0*\y) {$\affmod_5$};
        \node (M) at (4*\x, 0*\y) {$\scriptstyle{\dim = 8}$};
        \node (R1) at (4*\x, 1*\y) {$\scriptstyle{\dim = 4}$};
        \node (A4_2) at (4*\x, 2*\y) {$\scriptstyle{\dim = 2}$};
        \node (A4_3) at (4*\x, 3*\y) {$\scriptstyle{\dim = 0}$};
        \path (A1_0) edge node [auto] {$\scriptstyle{}$} (A1_1);
        \path (A1_1) edge node [auto] {$\scriptstyle{}$} (A1_2);
        \path (A1_3) edge node [auto] {$\scriptstyle{}$} (A0_2);
        \path (A1_2) edge node [auto] {$\scriptstyle{}$} (A1_3);
        \path (A0_2) edge node [auto] {$\scriptstyle{}$} (A1_0);
        \node (B1_3) at (2*\x, 3*\y) {$\scriptstyle{\bZ_5^2 \ltimes \bZ_4}$};
        \path (A1_3) edge [->] (B1_3);
        \node (B0_2) at (-1*\x, 2*\y) {$\scriptstyle{\bZ_5}$};
        \path (A0_2) edge [->] (B0_2);
        \node (B1_2) at (2*\x, 2*\y) {$\scriptstyle{\bZ_4}$};
        \path (A1_2) edge [->] (B1_2);
        \node (B1_1) at (2*\x, 1*\y) {$\scriptstyle{\bZ_2}$};
        \path (A1_1) edge [->] (B1_1);
        \node (B1_0) at (2*\x, 0*\y) {$\scriptstyle{\{1\}}$};
        \path (A1_0) edge [->] (B1_0);
        \end{tikzpicture}
        \caption{Hasse diagram for $\nu = 5$.}\label{fig:z5}
      \end{figure}

      \begin{remark}\label{remark:origin is godeaux}
        We worked in $\affmod_5$; it may happen that some of them lie in the locus of $\affmod_5$ we have to wipe out because of bad singularities. This is not the case for $\nu = 5$: we know that the origin $O$ represents a Godeaux surface (actually, the one Godeaux himself constructed). Hence, the space of Godeaux surfaces is a nonempty open set in every subspace we consider, since each one passes through the origin.
      \end{remark}

      It is easy to see that the high number of components in the three middle cases are due to the fact that up to now we are considering $\affmod_5$ and not $\schmod_5$ itself. Indeed, passing to the quotient, all the components collapse in $\schmod_5$ to one irreducible component for each group.
      
    \subsection{Torsion of order four}
      The results for the case $\nu = 4$ are given in Table~\ref{tab:z4} (where $s, v \in \bZ_4^\star$). In this case, the origin does not represent anymore a Godeaux surface. Indeed, in the origin we have $q_2 = y_1^2 + y_3^2$ which is reducible. Therefore, the argument of Remark~\ref{remark:origin is godeaux} does not apply. We will show later that the three vector spaces with a bullet on the right are exactly the ones not containing Godeaux surfaces.  
      
      \begin{table}[h]
        \begin{tabular}{lccllccc}
          \toprule
          Group & \name{GAP} id & V. sp. & \multicolumn{2}{c}{Equations} & Dim. & Comp.\\
          \midrule[\heavyrulewidth]
          $\{1\}$ & $(1, 1)$ & $\affmod_4$ & & & $8$ & $1$\\
          \midrule
          \multirow{5}{*}{$\bZ_2$} & \multirow{5}{*}{$(2, 1)$} & $R_1$ & \multicolumn{2}{l}{$b_v = 0$, $\forall v$} & $6$ & $1$\\
          & & \multirow{2}{*}{$W_s$} & \multicolumn{2}{l}{$b_1 = b_3 \xi^s$}  & \multirow{2}{*}{$5$} & \multirow{2}{*}{$2$}\\
          & & & $c_1 = c_3 \xi^{2s}$ & $d_1 = d_3 \xi^{2s}$\\
          & & $R_2$ & \multicolumn{2}{l}{$a^\prime = b_3 = c_v = 0$, $\forall v$} & $4$ & $1$\\
          & & $R_3$ & \multicolumn{2}{l}{$a^\prime = b_1 = c_v = 0$, $\forall v$} & $4$ & $1$\\
          \midrule
          \multirow{5}{*}{$\bZ_2^2$} & \multirow{5}{*}{$(4, 2)$} & \multirow{2}{*}{$R_{4, s}$} & \multicolumn{2}{l}{$b_v = 0$, $\forall v$} & \multirow{2}{*}{$4$} & \multirow{2}{*}{$2$}\\
          & & & $c_1 = c_3 \xi^s$ & $d_1 = d_3 \xi^s$\\
          & & \multirow{2}{*}{$R_{5, s}$} & \multicolumn{2}{l}{$a^\prime = b_1 = b_3 = 0$} & \multirow{2}{*}{$3$} & \multirow{2}{*}{$2$}\\
          & & & $c_1 = c_3 \xi^s$ & $d_1 = d_3 \xi^{s+2}$\\
          & & $S_5$ & \multicolumn{2}{l}{$a^\prime = b_v = c_v = 0$, $\forall v$} & $3$ & $1$\\
          \midrule
          \multirow{3}{*}{$\bZ_4$} & \multirow{3}{*}{$(4, 1)$} & $S_4$ & \multicolumn{2}{l}{$a^\prime = b_v = d_v = 0$, $\forall v$} & $3$ & $1$\\
          & & $S_2$ & \multicolumn{2}{l}{$a = a^\prime = b_3 = c_v = d_3 = 0$, $\forall v$} & $2$ & $1$\\
          & & $S_3$ & \multicolumn{2}{l}{$a = a^\prime = b_1 = c_v = d_1 = 0$, $\forall v$} & $2$ & $1$\\
          \midrule
          \multirow{6}{*}{$\bZ_4 \times \bZ_2$} & \multirow{6}{*}{$(8, 2)$} & \multirow{2}{*}{$S_{6, s}$} & \multicolumn{2}{l}{$a^\prime = b_v = d_v = 0$, $\forall v$} & \multirow{2}{*}{$2$} & \multirow{2}{*}{$2$}\\
          & & & \multicolumn{2}{l}{$c_1 = c_3 \xi^s$}\\
          & & \multirow{2}{*}{$T_{4, s}$} & \multicolumn{2}{l}{$a = a^\prime = b_v = d_v = 0$, $\forall v$} & \multirow{2}{*}{$1$} & \multirow{2}{*}{$2$}\\
          & & & \multicolumn{2}{l}{$c_1 = c_3 \xi^{s+1}$}\\
          & & $T_2$ & \multicolumn{2}{l}{$a = a^\prime = b_v = c_v = d_3 = 0$, $\forall v$} & $1$ & $1$\\
          & & $T_3$ & \multicolumn{2}{l}{$a = a^\prime = b_v = c_v = d_1 = 0$, $\forall v$} & $1$ & $1$\\
          \midrule
          \multirow{3}{*}{$D_8$} & \multirow{3}{*}{$(8, 3)$} & $S_1$ & \multicolumn{2}{l}{$b_v = c_v = d_v = 0$, $\forall v$} & $2$ & $1$ & $\bullet$\\
          & & \multirow{2}{*}{$S_{7, s}$} & \multicolumn{2}{l}{$a^\prime = b_v = c_v = 0$, $\forall v$} & \multirow{2}{*}{$2$} & \multirow{2}{*}{$2$}\\
          & & & \multicolumn{2}{l}{$d_1 = d_3 \xi^s$}\\
          \midrule
          $(\bZ_4 \times \bZ_2) \ltimes \bZ_2$ & $(16, 13)$ & $T_1$ & \multicolumn{2}{l}{$a^\prime = b_v = c_v = d_v = 0$, $\forall v$} & $1$ & $1$ & $\bullet$\\
          \midrule
          $\bZ_4^2 \ltimes \bZ_2$ & $(32, 11)$ & $O$ & \multicolumn{2}{l}{$a = a^\prime = b_v = c_v = d_v = 0$, $\forall v$} & $0$ & $1$ & $\bullet$\\
          \bottomrule
        \end{tabular}
        \caption{Special subcomponents in the case $\nu = 4$.}\label{tab:z4}
      \end{table}

      Again, when a space has several components, they collapse to one in $\schmod_4$; moreover, we can easily check that $R_2$ and $R_3$ collapse into one irreducible component inside $\schmod_4$; the same is true for $S_2$ and $S_3$, and $T_2$ and $T_3$. We define $R_{2,3} = R_2 \cup R_3$, $S_{2,3} = S_2 \cup S_3$ and $T_{2,3} = T_2 \cup T_3$.
      
      As we did before, we represent all the vector spaces into Figure~\ref{fig:z4}, ordered by containment. We recall that a vertical path means containment, but here we have also vertical dashed segments: for example, the one connecting $T_{4,s}$ with $S_{6,s}$ means that the former is not contained in the latter, but it is so when seen in the quotient $\schmod_4$. Also, vector spaces with dashed circle are the same as marked vector spaces in the table (that is, they do not contain any point representing Godeaux surfaces). To prove that they are exactly the spaces not containing Godeaux surfaces, we ``climb'' the diagram starting from the origin $O$, and for each space we check if it contains some (equivalently, an open subset of) Godeaux surfaces.

      \begin{figure}[h]
        \begin{tikzpicture}
        \def\x{1.3}
        \def\y{-1.2}
        \node[shape=circle,draw,dashed,inner sep=0.03cm] (O) at (4*\x, 7*\y) {$O$};
        \node[shape=circle,draw,dashed,inner sep=0.03cm] (T1) at (2*\x, 6*\y) {$T_1$};
        \node[shape=circle,draw,dashed,inner sep=0.03cm] (S1) at (1*\x, 5*\y) {$S_1$};
        \node[shape=circle,draw,inner sep=0.03cm] (T4) at (4*\x, 6*\y) {$T_{4,i}$};
        \node[shape=circle,draw,inner sep=0.03cm] (T2) at (6*\x, 6*\y) {$T_{2,3}$};
        \node[shape=circle,draw,inner sep=0.03cm] (S6) at (4*\x, 5*\y) {$S_{6,i}$};
        \node[shape=circle,draw,inner sep=0.03cm] (S7) at (3*\x, 5*\y) {$S_{7,i}$};
        \node[shape=circle,draw,inner sep=0.03cm] (S2) at (6*\x, 5*\y) {$S_{2,3}$};
        \node[shape=circle,draw,inner sep=0.03cm] (R5) at (3*\x, 4*\y) {$R_{5,i}$};
        \node[shape=circle,draw,inner sep=0.03cm] (S4) at (4*\x, 4*\y) {$S_4$};
        \node[shape=circle,draw,inner sep=0.03cm] (S5) at (5*\x, 4*\y) {$S_5$};
        \node[shape=circle,draw,inner sep=0.03cm] (R4) at (3*\x, 3*\y) {$R_{4,i}$};
        \node[shape=circle,draw,inner sep=0.03cm] (R2) at (6*\x, 3*\y) {$R_{2,3}$};
        \node[shape=circle,draw,inner sep=0.03cm] (W) at (3*\x, 2*\y) {$W_i$};
        \node[shape=circle,draw,inner sep=0.03cm] (R1) at (4*\x, 1*\y) {$R_1$};
        \node[shape=circle,draw,inner sep=0.03cm] (M) at (4*\x, 0*\y) {$\affmod_4$};
        \node (A9_0) at (8*\x, 0*\y) {$\scriptstyle{\dim = 8}$};
        \node (A9_1) at (8*\x, 1*\y) {$\scriptstyle{\dim = 6}$};
        \node (A9_2) at (8*\x, 2*\y) {$\scriptstyle{\dim = 5}$};
        \node (A9_3) at (8*\x, 3*\y) {$\scriptstyle{\dim = 4}$};
        \node (A9_4) at (8*\x, 4*\y) {$\scriptstyle{\dim = 3}$};
        \node (A9_5) at (8*\x, 5*\y) {$\scriptstyle{\dim = 2}$};
        \node (A9_6) at (8*\x, 6*\y) {$\scriptstyle{\dim = 1}$};
        \node (A9_7) at (8*\x, 7*\y) {$\scriptstyle{\dim = 0}$};
        \path (O) edge node [auto] {$\scriptstyle{}$} (T1);
        \path (O) edge node [auto] {$\scriptstyle{}$} (T2);
        \path (O) edge node [auto] {$\scriptstyle{}$} (T4);
        \path (S7) edge node [auto] {$\scriptstyle{}$} (S5);
        \path (T1) edge node [auto] {$\scriptstyle{}$} (S1);
        \path (S1) edge [bend left=40] node [auto] {$\scriptstyle{}$} (R1);
        \path (T1) edge node [auto] {$\scriptstyle{}$} (S6);
        \path (T1) edge node [auto] {$\scriptstyle{}$} (S7);
        \path (T4) edge [dashed] node [auto] {$\scriptstyle{}$} (S6);
        \path (S5) edge node [auto] {$\scriptstyle{}$} (R1);
        \path (R4) edge node [auto] {$\scriptstyle{}$} (R1);
        \path (S5) edge node [auto] {$\scriptstyle{}$} (R2);
        \path (S6) edge node [auto] {$\scriptstyle{}$} (R5);
        \path (T2) edge node [auto] {$\scriptstyle{}$} (S5);
        \path (S7) edge node [auto] {$\scriptstyle{}$} (R5);
        \path (T2) edge node [auto] {$\scriptstyle{}$} (S2);
        \path (S2) edge node [auto] {$\scriptstyle{}$} (R2);
        \path (R4) edge node [auto] {$\scriptstyle{}$} (W);
        \path (S6) edge node [auto] {$\scriptstyle{}$} (S4);
        \path (R5) edge [dashed] node [auto] {$\scriptstyle{}$} (R4);
        \path (S4) edge node [auto] {$\scriptstyle{}$} (R1);
        \path (R1) edge node [auto] {$\scriptstyle{}$} (M);
        \path (W) edge node [auto] {$\scriptstyle{}$} (M);
        \path (R2) edge [bend right=10] node [auto] {$\scriptstyle{}$} (M);

        \node (Z442) at (3*\x, 7.2*\y) {$\scriptstyle{\bZ_4^2 \ltimes \bZ_2}$};
        \path (O) edge [->] (Z442);
        \node (Z422) at (1.4*\x, 6.8*\y) {$\scriptstyle{(\bZ_4 \times \bZ_2) \ltimes \bZ_2}$};
        \path (T1) edge [->] (Z422);
        \node (Z42a) at (6.5*\x, 6.7*\y) {$\scriptstyle{\bZ_4 \times \bZ_2}$};
        \path (T2) edge [->] (Z42a);
        \node (Z42b) at (4.8*\x, 5.5*\y) {$\scriptstyle{\bZ_4 \times \bZ_2}$};
        \path (S6) edge [->] (Z42b);
        \path (T4) edge [->] (Z42b);
        \node (Z22a) at (2.4*\x, 3.5*\y) {$\scriptstyle{\bZ_2^2}$};
        \path (R5) edge [->] (Z22a);
        \path (R4) edge [->] (Z22a);
        \node (Z22c) at (4.8*\x, 4.7*\y) {$\scriptstyle{\bZ_2^2}$};
        \path (S5) edge [->] (Z22c);
        \node (D8) at (2*\x, 4.5*\y) {$\scriptstyle{D_8}$};
        \path (S7) edge [->] (D8);
        \path (S1) edge [->] (D8);
        \node (Z4) at (6.5*\x, 4.2*\y) {$\scriptstyle{\bZ_4}$};
        \path (S2) edge [->] (Z4);
        \node (Z4b) at (4.4*\x, 3.5*\y) {$\scriptstyle{\bZ_4}$};
        \path (S4) edge [->] (Z4b);
        \node (Z2a) at (6.2*\x, 2.4*\y) {$\scriptstyle{\bZ_2}$};
        \path (R2) edge [->] (Z2a);
        \node (Z2b) at (3.5*\x, 1.5*\y) {$\scriptstyle{\bZ_2}$};
        \path (R1) edge [->] (Z2b);
        \path (W) edge [->] (Z2b);
        \node (Z1) at (5*\x, 0.2*\y) {$\scriptstyle{\{1\}}$};
        \path (M) edge [->] (Z1);
        \end{tikzpicture}
        \caption{Hasse diagram for $\nu = 4$.}\label{fig:z4}
      \end{figure}
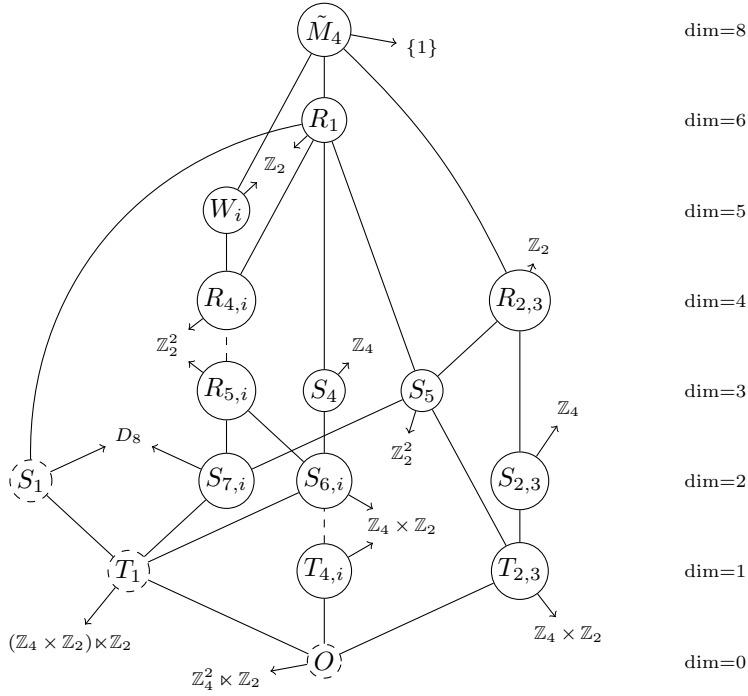

      We already seen that $O$ cannot corresponds to a Godeaux surface, since $q_2 = y_1^2 + y_3^2$ is reducible. For the same reason, $T_1$ and $S_1$ do not contain points corresponding to Godeaux surfaces.
      
      In $T_2$ we have these equations:
      \begin{align*}
        q_0 &= x_1^4 + x_2^4 + x_3^4 + a x_1^2 x_3^2 + y_1 y_3\text{,}\\
        q_2 &= d_1 x_1^2 x_2^2 + y_1^2 + y_3^2\text{;}
      \end{align*}
      dehomogenizing by $x_i$, $i \in \{1,2,3\}$, we obtain an affine covering of $X$, so we can compute the singularities via the Jacobian matrix. For example, in the open set $(x_1 \neq 0) \cong \mathbb{A}^4$ we have \[ J = \begin{pmatrix} 4 x_2^3 & 4 x_3^3 & y_3 & y_1\\ 2 d_1 x_2 & 0 & 2 y_1 & 2 y_3 \end{pmatrix}\text{;} \]
      thanks to the last minor involving only the $y_i$, we have two cases in which $J$ has rank strictly less than $2$.
      \begin{enumerate}
        \item If $y_1 = y_3 = 0$, the other minors, involving only the $x_i$, have to be $0$, so we have the following equations: \[0 = 8 d_1 x_2 x_3^3 = 1 + x_2^4 + x_3^4 = d_1 x_2^2\text{.} \] Since we are interested in an open subset of $T_2$, we may assume $d_1 \neq 0$ so we get four singular points $x_2 = 0$, $x_3 = \xi^j \sqrt[4]{-1}$.
        \item If $y_3 = \pm y_1 \neq 0$, then the second row is plus or minus two times the first row; in particular we have \[ 0 = x_3 = 8 x_2^3 \mp 2 d_1 x_2 = 1 + x_2^4 \pm y_1^2 = d_1 x_2^2 + 2 y_1^2\text{.} \] It cannot happen that $x_2 = 0$, so $x_2^2 = d_1 / 4$ and $y_1^2 = - d_1^2 / 8$; but this implies $d_1^2 = 16$ and we can discard this particular situation that happens only in a proper closed subset of $T_2$.
      \end{enumerate}
      In the same way we can find singular points in the other two affine open subsets $(x_2 \neq 0)$ and $(x_3 \neq 0)$, and the result is that we have $8$ singular points for the surface $X$ represented by a generic point of $T_2$: \begin{align*} &[1, 0, \xi^j \sqrt[4]{-1}, 0, 0]\text{,} & &[0, 1, \xi^j \sqrt[4]{-1}, 0, 0]\text{.}\end{align*} Now we have to check if these are rational double points or worse. For example, consider the point $p = [1, 0, \xi^j \sqrt[4]{-1}, 0, 0]$ in the affine open set relative to $x_1$, we have ${(\partial q_0 / \partial x_3)}|_p \neq 0$, hence we can represent, analitically locally, $x_3$ as $x_3(p) + g(x_2, y_1, y_3)$. Substituting $x_3$ in $q_2$, we obtain the expression \[ q_2 = y_1^2 + y_3^2 + x_3(p) d_1 x_2^2 + \cdots \] where $x_3(p) \neq 0$ and the other terms are of order at least three in $p$. So the singularity is of type $A_1$, in particular it is a rational double point. The situation is the same for every other singular point (since they're in the same $G$-orbit), so we conclude that in $T_2$ there is a nonempty open set of Godeaux surfaces.
      
      The situation in $T_3$ is completely specular. We do not write the similar computation for $T_{4,i}$ and $S_{7,i}$ anyway both contain an open subset of Godeaux surfaces.
    
    \subsection{Torsion of order three}    
      The results given from the program in the case $\nu = 3$ are simpler then the others. This is understandable: going from $\bZ_5$ to $\bZ_4$ we've seen an increasing complexity on the vector spaces, but a decreasing order of the automorphisms groups. In this last case, the latter behaviour prevails on the former.

      \begin{table}[h]
        \begin{tabular}{lccllccc}
          \toprule
          Group & \name{GAP} id & V. sp. & \multicolumn{2}{c}{Equations} & Dim. & Comp.\\
          \midrule[\heavyrulewidth]
          $\{1\}$ & $(1, 1)$ & $\affmod_3$ & & & $9$ & $1$\\
          \midrule
          $\bZ_2$ & $(2, 1)$ & $A$ & \multicolumn{2}{l}{$a_1 = a_2 = b_1 = b_2 = 0$} & $5$ & $1$ & $\bullet$\\
          \bottomrule
        \end{tabular}
        \caption{Special subcomponents in the case $\nu = 3$.}\label{tab:z3}
      \end{table}     
      
      Indeed, the results listed in Table~\ref{tab:z3} are just two lines, the second of them describing a vector space not containing any Godeaux surface (we already use that for a point to describe a Godeaux surface, it must be $a_1 + a_2 \neq 0$); i.e.\ Godeaux surfaces with torsion of order three have no nontrivial automorphisms.

  \section{Moduli stacks}\label{sec:stacks}
    In this section we define the moduli stack $\trumod_\nu$ of Godeaux surfaces with torsion of order $\nu$ and relate it to the computation of automorphisms of the previous section. More precisely, let $G \cong \bZ_\nu$ be the torsion group of a Godeaux surface $S$ realized as a subgroup of $\Aut(\bP)$ as in~\ref{subsection:five}, $H_\nu = \N_{\Aut(\bP)}(G)$ as in Remark~\ref{remark:finite_group}, and denote with $\stamod_\nu$ the quotient stack $[\affmod_\nu / (H_\nu/G)]$. We will show that there is a natural map $\Phi\colon \stamod_\nu \to \trumod_\nu$ and that it is an equivalence on points. Moreover, we will show that this map is an isomorphism in the case $\nu = 5$; there are no reasons to doubt that this holds also for the other torsions. Nevertheless we would need finer arguments, since the description of the canonical model of a surface with lower torsion is not as nice as in the case $\nu = 5$.
    
    \begin{remark}
      There are two natural definitions for the moduli stack of surfaces: the first considers flat projective families where the fibers are smooth minimal models of some surface in the class; the second considers canonical models instead of minimal models. Often the latter seems more natural than the former and here we will pursue this approach. Recall that for a Godeaux surface $S$, we denoted with $X \to S$ the smooth cover coming from $\Tors(S)$, and with $\obar{X}$ and $\obar{S}$ the canonical models of $X$ and $S$.
    \end{remark}

    \begin{definition}
      The \emph{moduli stack of Godeaux surfaces\/} with torsion of order $\nu$ is the stack $\trumod_\nu$ defined as a category fibered in groupoids by:
      \begin{align*}
        \Obj(\trumod_\nu) &= \left\{ \pi\colon \obar{S}_B \to B \,\middle|\, \begin{array}{@{}l@{}} \text{$\pi$ flat, projective,}\\ \text{$\forall b \in B$, $\obar{S}_b$ is the canonical}\\ \text{model of a Godeaux surface}\\ \text{with torsion of order $\nu$} \end{array} \right\}\text{, and}\\
        \Mor_{\trumod_{\nu}}(\pi, \pi^\prime) &= \left\{ (\phi, \psi) \,\middle|\,
          \begin{tikzpicture}[baseline]
            \def\x{1.5}
            \def\y{-1.2}
            \node (A0_0) at (0*\x, -.5*\y) {$\obar{S}_B$};
            \node (A0_1) at (1*\x, -.5*\y) {$\obar{S}^\prime_{B^\prime}$};
            \node (A1_0) at (0*\x, .5*\y) {$B$};
            \node (A1_1) at (1*\x, .5*\y) {$B^\prime$};
            \node (B1_1) at (.5*\x, 0*\y) {$\square$};
            \path (A0_0) edge [->] node [auto,swap] {$\scriptstyle{\psi}$} (A0_1);
            \path (A1_0) edge [->] node [auto] {$\scriptstyle{\phi}$} (A1_1);
            \path (A0_1) edge [->] node [auto,swap] {$\scriptstyle{\pi^\prime}$} (A1_1);
            \path (A0_0) edge [->] node [auto] {$\scriptstyle{\pi}$} (A1_0);
          \end{tikzpicture}
        \right\}\text{;}
      \end{align*}
      the projection to schemes sends $\pi\colon \obar{S}_B \to B$ to $B$ and $(\phi, \psi)$ to $\phi$.
    \end{definition}

    \begin{proposition}
      There exists a natural morphism of stacks $\Phi\colon \stamod_\nu \to \trumod_\nu$.
    \end{proposition}

    \begin{proof}
      Let $\tilde{\Phi} \colon \affmod_\nu \to \trumod_\nu$ be the morphism determined by the universal family $U \to \affmod_\nu$, with $U \subseteq \affmod_\nu \times (\bP / G)$. We will prove that $\tilde{\Phi}$ is $H_\nu$-equivariant, and so it passes to the quotient (recall that $G \subset H_\nu$ acts trivially on $\affmod_\nu$).
      
       Being $H_\nu$-equivariant means that for every $h \in H_\nu$ we have a canonical $2$-morphism $\eta$ making the following diagram $2$-commutative:
        \[
        \begin{tikzpicture}
          \def\x{1.2}
          \def\y{-1.0}
          \node (A0_0) at (150:\x) {$\affmod_\nu$};
          \node (A0_2) at (30:\x) {$\affmod_\nu$};
          \node (A1_1) at (270:1) {$\trumod_\nu$.};
          \path (A0_0) edge [->] node [auto,swap] {$\scriptstyle{\tilde{\Phi}}$} (A1_1);
          \path (A0_2) edge [->] node [auto] {$\scriptstyle{\tilde{\Phi}}$} (A1_1);
          \path (A0_0) edge [->] node [auto] {$\scriptstyle{h}$} (A0_2);
          \path (210:\x/2) -- (A0_2)
            node [pos=0.5,allow upside down,sloped] {$\Longrightarrow$}
            node [pos=0.4,auto,swap] {$\scriptstyle{\eta}$};
        \end{tikzpicture}
        \]
        Given a map $f\colon T \to \affmod_\nu$, we have that $\tilde{\Phi}(f)$ is the family $\obar{S}_T \to T$ in the cartesian diagram
        \[
        \begin{tikzpicture}
          \def\x{1.5}
          \def\y{-1.2}
          \node (A0_0) at (0*\x, 0*\y) {$\obar{S}_T$};
          \node (A0_1) at (1*\x, 0*\y) {$U$};
          \node (A1_0) at (0*\x, 1*\y) {$T$};
          \node (A1_1) at (1*\x, 1*\y) {$\affmod_\nu$};
          \node (A2_2) at (0.5*\x, 0.5*\y) {$\square$};
          \path (A0_0) edge [->] node [auto] {$\scriptstyle{}$} (A0_1);
          \path (A1_0) edge [->] node [auto,swap] {$\scriptstyle{f}$} (A1_1);
          \path (A0_1) edge [->] node [auto] {$\scriptstyle{u}$} (A1_1);
          \path (A0_0) edge [->] node [auto,swap] {$\scriptstyle{\tilde{\Phi}(f)}$} (A1_0);
        \end{tikzpicture}
        \]
        and in the same way $\tilde{\Phi} \circ h(f)$ is the family $\obar{S}_T^\prime \to T$. We have to define $\eta(f)\colon \tilde{\Phi}(f) \Rightarrow \tilde{\Phi} \circ h(f)$ as a couple of morphisms $(g, \bar{g})$ making the following diagram cartesian:
        \[
        \begin{tikzpicture}
          \def\x{1.5}
          \def\y{-1.2}
          \node (A0_0) at (0*\x, 0*\y) {$\obar{S}_T$};
          \node (A0_1) at (1*\x, 0*\y) {$\obar{S}_T^\prime$};
          \node (A1_0) at (0*\x, 1*\y) {$T$};
          \node (A1_1) at (1*\x, 1*\y) {$T$};
          \node (A2_2) at (0.5*\x, 0.5*\y) {$$};
          \path (A0_0) edge [->] node [auto] {$\scriptstyle{\bar{g}}$} (A0_1);
          \path (A1_0) edge [->] node [auto,swap] {$\scriptstyle{g}$} (A1_1);
          \path (A0_1) edge [->] node [auto] {$\mathrlap{\scriptstyle{\tilde{\Phi}(h(f))}}$} (A1_1);
          \path (A0_0) edge [->] node [auto,swap] {$\scriptstyle{\tilde{\Phi}(f)}$} (A1_0);
        \end{tikzpicture}
        \]
        Since \[ \obar{S}_T = T \times_{\affmod_\nu} U \subseteq T \times_{\affmod_\nu} (\affmod_\nu \times (\bP/G)) \cong T \times (\bP/G)\text{,}\] over every point $t \in T$ we have the natural isomorphism $h\colon \obar{S}_{T,t} \to \obar{S}_{T, t}^\prime$, that extends to $\psi\colon \obar{S}_T \to \obar{S}^\prime_T$, and we define $\eta(f) = (\id_T, \psi)$.\qedhere
    \end{proof}

    \begin{lemma}\label{lemma:isomorphism on points}
      The morphism $\Phi$ induces an equivalence of groupoids $\Phi(\bC)\colon \stamod_\nu(\bC) \to \trumod_\nu(\bC)$.
    \end{lemma}

    \begin{proof}
      An object of $\stamod_\nu(\bC)$ is a diagram
      \[
      \begin{tikzpicture}
        \def\x{1.5}
        \def\y{-1.2}
        \node (A0_0) at (0*\x, 0*\y) {$T$};
        \node (A0_1) at (1*\x, 0*\y) {$\affmod_\nu$};
        \node (A1_0) at (0*\x, 1*\y) {$\Spec \bC$};
        \path (A0_0) edge [->] node [auto] {$\scriptstyle{f}$} (A0_1);
        \path (A0_0) edge [->] node [auto,swap] {$\scriptstyle{\pi}$} (A1_0);
      \end{tikzpicture}
      \]
      with $\pi$ an $(H_\nu/G)$-torsor and $f$ an $(H_\nu/G)$-equivariant morphism; in other words, \[ \stamod_\nu(\bC) = \left\{(T, f) \,\middle|\, \begin{array}{@{}l@{}} \text{$T \cong H_\nu/G$ as schemes}\\ \text{$H_\nu/G$ acts freely on $T$}\\ \text{$f$ $(H_\nu/G)$-equivariant}\end{array}\right\} \] As a consequence, all the points of $T$ are mapped to points of $\affmod_\nu$ corresponding to the same Godeaux surface, modulo isomorphism, that is the image of the object via $\Phi(\bC)$.
      
      We will prove that $\Phi(\bC)$ is essentially surjective and that is bijective on morphisms.
      \begin{description}
      \item[Essentially surjective] we have to prove that for every Godeaux surface $S$ there exists a point in $\stamod_\nu(\bC)$ sent to a surface isomorphic to the canonical model of $S$. The object $(H_\nu/G, f)$ will do if $f(e) \in \affmod_\nu$ is a point corresponding to $S$, and $f$ is extended equivariantly.
      \item[Bijection on morphisms] we have to prove that automorphisms of $(T, f)$ are in a bijection with automorphisms of $\obar{S} = \Phi(\bC)(T, f)$; this is exactly what we proved in the previous section, since automorphisms of $(T,f)$ are in a bijection with stabilizers of $H_\nu/G$ over a point $f(t)$ for $t \in T$ (this does not change when $t$ changes since all $f(t)$ are in the same orbit).\qedhere
      \end{description}
    \end{proof}    

    We recall here a useful statement.
    
    \begin{lemma}\label{lemma:isomorphism criterion}
      Let $X$ and $Y$ be smooth stacks of dimension $d$; then a morphism $f\colon X \to Y$ is an isomorphism if and only if
      \begin{enumerate}
        \item\label{item:groupoid} $f(\Spec \bC)$ is an equivalence of groupoids, and
        \item\label{item:tangents} $f$ is bijective on tangent vectors.
      \end{enumerate}
    \end{lemma}

    \begin{remark}\label{remark:chi}
      For a Godeaux surface $S$, Riemann-Roch yields \[ \chi(\T_S) = 2 \sK_S^2 - 10 \chi(\sO_S) = -8\text{,} \] hence $\ho^2(S, \T_S) = 0$ if and only if $\ho^1(S, \T_S) = 8$. If $S$ is a Godeaux surface with singular canonical model, we can still define the Euler characteristic of the pair  $(\Omega_{\obar{S}}, \sO_{\obar{S}})$ to be \[ \chi(\Omega_{\obar{S}}, \sO_{\obar{S}}) = \sum_{i=0}^2 \ext^i(\Omega_{\obar{S}}, \sO_{\obar{S}})\text{,} \] generalizing the previous one. We know that $S$ can be deformed to a Godeaux surface $S^\prime$ with smooth canonical model; since the dimensions of the $\Ext$ groups are deformation invariants, the previous computation ensures $\chi(\Omega_{\obar{S}}, \sO_{\obar{S}}) = 8$. Moreover, since $\obar{S}$ is of general type, $\ext^0(\Omega_{\obar{S}}, \sO_{\obar{S}}) = 0$, and we obtain again \[ \ext^1(\Omega_{\obar{S}}, \sO_{\obar{S}}) = 8 \iff \ext^2(\Omega_{\obar{S}}, \sO_{\obar{S}}) = 0\text{.}\]
    \end{remark}

    \begin{remark}\label{remark:necessary for isomorphism}
      To show that $\Phi\colon \stamod_\nu \to \trumod_\nu$ is an isomorphism, it is enough to prove that for every Godeaux surface $S$ with torsion of order $\nu$, we have:
      \begin{enumerate}
        \item $\ext^1(\Omega_{\obar{S}}, \sO_{\obar{S}}) = 8$;
        \item $\Phi$ is bijective on tangent vectors.
      \end{enumerate}
      
      Clearly $\stamod_\nu$ is smooth of dimension $8$. The first condition ensures, by Remark~\ref{remark:chi}, that also the moduli stack $\trumod_\nu$ is so. Hence we can apply the criterion of Lemma~\ref{lemma:isomorphism criterion}: condition~\ref{lemma:isomorphism criterion}.\ref{item:groupoid} is already proved in Lemma~\ref{lemma:isomorphism on points}, while condition~\ref{lemma:isomorphism criterion}.\ref{item:tangents} is the second requirement listed here.
    \end{remark}
    
    We will prove the two conditions in the case $\nu = 5$. In the following, we will write ``Godeaux surface'' for ``Godeaux surface with torsion of order $5$''.

    \begin{lemma}\label{lemma:vanishing of group}
      Let $\obar{X} \subseteq \bP^3$ be a quintic hypersurface with at most RDP as singularities; then $\Ho^1(\obar{X}, \T_{\bP^n}|_{\obar{X}}) $ vanishes.
    \end{lemma}

    \begin{proof}
      We will prove that $\Ho^1(\obar{X}, \T_{\bP^3}|_{\obar{X}})\spcheck = 0$. By Serre duality and since by adjunction $\omega_{\obar{X}} = \sO_{\obar{X}}(1)$, this is equal to $\Ho^1(\obar{X}, \Omega_{\bP^3}|_{\obar{X}}(1))$.

      From the cohomology of the Euler sequence tensored with $\sO_{\obar{X}}(1)$, we get \[ \Ho^0(\obar{X}, \sO_{\obar{X}}^{\oplus (n+1)}) \to \Ho^0(\obar{X}, \sO_{\obar{X}}(1)) \to \Ho^1(\obar{X}, \Omega_{\bP^3}|_{\obar{X}}(1)) \to \Ho^1(\obar{X}, \sO_{\obar{X}}^{\oplus(n+1)})\text{;}\] the first map is surjective, while the last group is equal to ${\Ho^1(\obar{X}, \sO_{\obar{X}})}^{\oplus(n+1)} = 0$ since $\irr(\obar{X}) = 0$. Hence, $\Ho^1(\obar{X}, \Omega_{\bP^3}|_{\obar{X}}(1)) = 0$.\qedhere
    \end{proof}

    \begin{lemma}\label{lemma:moduli smooth}
      The moduli stack $\trumod_5$ of Godeaux surfaces is smooth of dimension $8$.
    \end{lemma}

    \begin{proof}
      By Remark~\ref{remark:chi}, it is enough to show $\ext^1(\Omega_{\obar{S}}, \sO_{\obar{S}}) = 8$ for every $S$.

      Let $X \to S$ be the cover associated to the torsion of $S$; we have seen that $\obar{X}$, the canonical model of $X$, embeds in $\bP^3$ as a quintic hypersurface with at most RDP; also, $\Ext^i(\Omega_{\obar{S}}, \sO_{\obar{S}})$ is just the $\bZ_5$-invariant part of $\Ext^i(\Omega_{\obar{X}}, \sO_{\obar{X}})$. Applying the functor $\Hom(_, \sO_{\obar{X}})$ to \[ 0 \to \sO_{\obar{X}}(-5) \to \Omega_{\bP^3}|_{\obar{X}} \to \Omega_{\obar{X}} \to 0\text{,} \] we get the exact sequence
      \begin{multline}\label{eq:exact sequence}
        \Hom(\Omega_{\obar{X}}, \sO_{\obar{X}}) \to \Ho^0(\obar{X}, \T_{\bP^3}|_{\obar{X}}) \to \Ho^0(\obar{X}, \sO_{\obar{X}}(5)) \to\\
        \to \Ext^1(\Omega_{\obar{X}}, \sO_{\obar{X}}) \to \Ho^1(\obar{X}, \T_{\bP^3}|_{\obar{X}})\text{.}
      \end{multline}
      The first group is zero because $\obar{X}$ is of general type, while we already proved that the last one vanishes in Lemma~\ref{lemma:vanishing of group}. Therefore we have a short exact sequence and to compute $\ext^1(\Omega_{\obar{S}}, \sO_{\obar{S}})$ we observe that
      \begin{enumerate}
        \item $\Ho^0(\obar{X}, \T_{\bP^3}|_{\obar{X}})$ has the same dimension as the group $\bP\GL(4)$, $15$; its $\bZ_5$-invariant part has dimension $3$, since it parametrizes infinitesimal deformation of linear isomorphisms of $\bP^3$ commuting with the action of $G$, and these correspond to diagonal matrices;
        \item $\Ho^0(\obar{X}, \sO_{\obar{X}}(5))$ has dimension $\Ho^0(\bP^3, \sO_{\bP^3}(5)) - 1 = \binom{3 + 5}{5}-1 = 55$; as we saw before, $\ho^0(\bP^3, \sO_{\bP^3}(5))^{\bZ_5} = 12$, then $\ho^0(\obar{X}, \sO_{\obar{X}}(5))^{\bZ_5} = 11$.
      \end{enumerate}  
      In particular, we obtain that $\ext^1(\Omega_{\obar{S}}, \sO_{\obar{S}}) = \ext^1(\Omega_{\obar{X}}, \sO_{\obar{X}})^{\bZ_5} = 11 - 3 = 8$.\qedhere
    \end{proof}

    \begin{lemma}
      The morphism $\Phi$ is bijective on tangent vectors.
    \end{lemma}
    
    \begin{proof}
      Fix a Godeaux surface $S$. Then $\T_{\trumod_5,[\obar{S}]} = \Ext^1(\Omega_{\obar{S}}, \sO_{\obar{S}})$, while $\T_{\stamod_5,[S]} = \T_{\affmod_5,[S]}$, since the projection $\affmod_5 \to \stamod_5$ is an \'etale cover. The morphism between the tangent spaces induced by $\Phi$ is the restriction, first to the $\bZ_5$-invariant part, then to $\T_{\affmod_5,[S]}$, of the map \[\Ho^0(\obar{X}, \sO_{\obar{X}}(5)) \to \Ext^1(\Omega_{\obar{X}}, \sO_{\obar{X}}) \] in the exact sequence~\eqref{eq:exact sequence}.

      Let $f$ be the quintic polynomial defining $\obar{X}$; then $\Ho^0(\obar{X}, \sO_{\obar{X}}(5))^{\bZ_5}$, as the restriction of $\Ho^0(\bP^3, \sO_{\bP^3}(5))^{\bZ_5}$, consists of quintic polynomials invariant with respect to the action of $G$; but they can be interpreted also as infinitesimal deformations of $f$ inside the quintic polynomials invariant with respect to $G$. In the same spirit, $\Ho^0(\obar{X}, \T_{\bP^3}|_{\obar{X}})^{\bZ_5}$ is the space of infinitesimal deformations of the identity matrix inside the matrices invariants with respect to $G$, that are the diagonal matrices.
      
      In other words, an element of $\Ho^0(\obar{X}, \T_{\bP^3}|_{\obar{X}})^{\bZ_5}$ can be represented by an infinitesimal deformation $I + \epsilon A$ with $A$ a diagonal matrix, modulo multiplication with scalars; to this, we associate an infinitesimal deformation of polynomials $(I + \epsilon A) f$, represented by the polynomial $Af$ in the space $\Ho^0(X, \sO_{\bP^3}(5))^{\bZ_5}$. Since $A$ is diagonal, $Af$ has exactly the same monomials of $f$, only with different coefficients, and in particular it is of the form $\sum a_i x_i^5 + \cdots$ with $a_i \neq 0$, and therefore does not intersect $\T_{\affmod_5, [S]}$, that contains only monomials without the terms $x_i^5$.\qedhere
    \end{proof}

    The last two lemmas, in view of Remark~\ref{remark:necessary for isomorphism}, yield the following theorem.

    \begin{theorem}\label{theorem:moduli godeaux are the quotient}
      The morphism $\Phi\colon \stamod_5 \to \trumod_5$ is an isomorphism of stacks.
    \end{theorem}

  \section{Inertia stacks}\label{sec:inertia}

    In this section we will compute the inertia stack of $\stamod_{\nu}$ for $\nu \in \{3,4,5\}$. Since $\stamod_3$ has trivial automorphisms groups (i.e., it is an algebraic space), we will work only on $\stamod_4$ and $\stamod_5$. These are quotients of an open subscheme of $\bA^8$ by an explicit finite group of projective matrices. Hence we can work out the components of the inertia stacks $\I(\stamod_4)$ and $\I(\stamod_5)$ from these representations.
    
    \subsection{Torsion of order five} 
      Let us have a look at Figure~\ref{fig:z5}; our problem is to identify automorphisms of Godeaux surfaces lying in different subspaces of $\affmod_5$. In the following, we will denote a generic surface in $\affmod_5$ as $S_{\affmod_5}$; in the same way, we define $S_Q$, $S_P$, $S_H$, $S_O$.
      
      In this case we do not need many computations: for example, there is a unique way to identify $\Aut(S_Q) \cong \bZ_2$ inside $\Aut(S_P) \cong \bZ_4$; the only ambiguities come up when we want to see where $\Aut(S_P) \cong \bZ_4$ and $\Aut(S_H) \cong \bZ_5$ goes inside $\Aut(S_O) \cong \bZ_5^2 \ltimes \bZ_4$. These are not actual problems, since to construct $\I(\stamod_5)$ we only need to see which automorphisms go to coincide when viewed in a larger group. Then it is obvious that only the identities will coincide in $\Aut(S_0)$, since the other automorphisms have different orders.
      
      In order to explain the general principle, we will give the computation even if it is not really necessary. In the following, we will write all groups $\Aut(S)$ as subgroup of the group $\Aut(S_0)$; this one is the quotient by $G \cong \bZ_5$ of the group $H_5 \cong \bZ_5^3 \ltimes \bZ_5^\star$. We will denote the matrix \[ \diag(1, \xi^{i_2}, \xi^{i_3}, \xi^{i_4}) P_{\sigma^h} \in H_\nu\text{,}\] where $\sigma$ is the permutation $(2,1,3,4)$, with $(i_2, i_3, i_4, h)$; $G$ lies inside $H_\nu$ as the subgroup generated by $(1,2,3,0)$. The same program we used to compute the abstract automorphisms groups gives us also the automorphisms groups embedded in $\bP\GL(4)$; in particular we obtain the following representations in $H_\nu/G$:
      \begin{align*}
        \Aut(S_{\stamod_5}) &= \langle(0,0,0,0)\rangle\text{,}\\
        \Aut(S_Q) &= \langle (0,0,0,2)\rangle\text{,}\\
        \Aut(S_P) &= \langle (0,0,0,1)\rangle\text{,}\\
        \Aut(S_H) &= \langle (0,0,1,0)\rangle\text{,}\\
        \Aut(S_O) &= \langle (1,0,0,0), (0,1,0,0), (0,0,0,1)\rangle\text{.}
      \end{align*}
      Note that these are the embedded automorphisms groups for just one component of $H$, $P$ and $Q$: indeed, if we do not choose carefully the components we may end with incompatible groups. We just have to do a choice of components that satisfies the Hasse diagram of containments even in $\affmod_5$.
      
      Once we have this explicit description, we know how automorphisms glue amongst different subschemes of $\affmod_5$, and we can write down the $100$ components of $\I(\stamod_5)$: 
      \begin{multline*}
        \I(\stamod_5) = (\stamod_5, (0,0,0,0)) \sqcup (Q, (0,0,0,2)) \sqcup \bigsqcup_{\mathclap{h \in \{1,3\}}}(P, (0,0,0,h)) \sqcup\\
        \sqcup \bigsqcup_{\mathclap{i \in \{1,2,3,4\}}} (H, (0,0,i,0)) \sqcup \bigsqcup_{\mathclap{i_1,i_2,i_3,h}} (O, (i_1,i_2,i_3,h))\text{,}
      \end{multline*}
      where the last union runs over all the $92$ elements of $\Aut(S_0)$ not previously considered. We can find automorphisms groups of all subcomponents of the components of the inertia stack by computing centralizers. For example, the automorphisms group of $O \subseteq (Q, (0,0,0,2))$ is the centralizer of $(0,0,0,2)$ inside $H_\nu/G$, that is $\bZ_5 \ltimes \bZ_5^\star$. It is easy to use \name{GAP} to compute all centralizers of each automorphism inside $H_\nu/G$ (we do not need the other centralizers since all other groups are abelian and so the centralizers are trivial).

      The following picture represents all the components of the inertia stack with all their stacky subcomponents (obviously with fake dimensions).
      
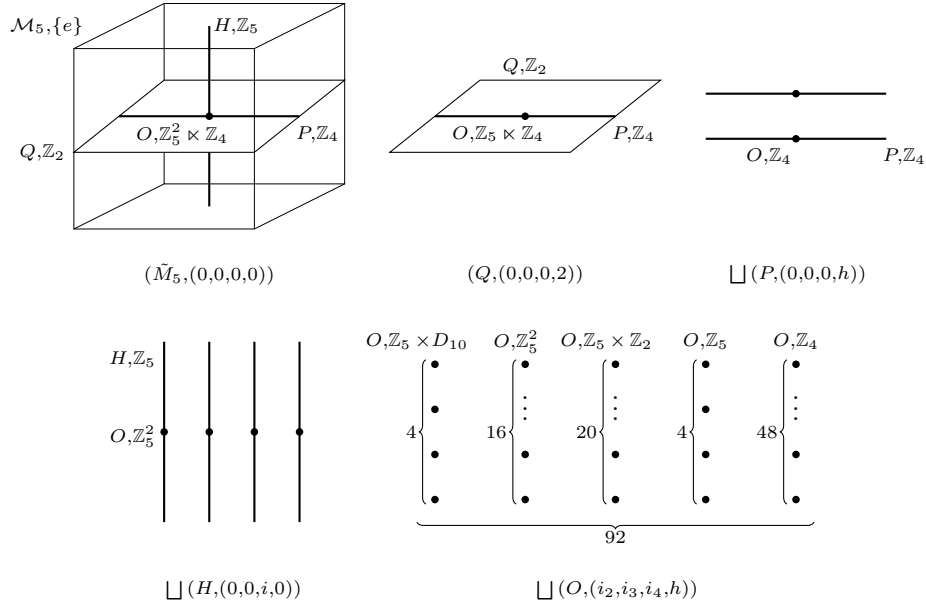
\[\begin{tikzpicture}
        \def\x{1.2};
        \def\y{.6};
        
        \draw[very thin] (1*\x,1*\y) -- (3*\x,1*\y) -- (3*\x,5*\y) -- (1*\x,5*\y) -- cycle; 
        \draw[very thin] (2*\x,4*\y) -- (3*\x,5*\y); 
        \draw[very thin] (2*\x,0*\y) -- (3*\x,1*\y); 
        \draw[very thin] (0*\x,0*\y) -- (1*\x,1*\y); 
        \draw[very thin] (0*\x,4*\y) -- (1*\x,5*\y); 
        \draw[fill=white] (2.5*\x,2.5*\y) -- (3*\x,3.3*\y) -- (1*\x,3.3*\y) -- (0.5*\x,2.5*\y); 
        \draw[thick] (1.5*\x,0.5*\y) -- (1.5*\x,4.5*\y); 
        \draw[fill=white] (0.5*\x, 2.5*\y) -- (0*\x,1.7*\y) -- (2*\x,1.7*\y) -- (2.5*\x,2.5*\y); 
        \draw[thick] (0.5*\x,2.5*\y) -- (2.5*\x,2.5*\y); 
        \fill (1.5*\x,2.5*\y) circle (.05); 
        \draw[very thin] (0*\x,0*\y) -- (2*\x,0*\y) -- (2*\x,4*\y) -- (0*\x,4*\y) -- cycle; 
        \path (1.5*\x,-1*\y) node {$\scriptstyle{(\affmod_5, (0,0,0,0))}$};
        \node at (-.3*\x,4.5*\y) {$\scriptstyle{\stamod_5, \{e\}}$};
        \node at (-.35*\x,1.7*\y) {$\scriptstyle{Q, \bZ_2}$};
        \node at (2.7*\x,2.1*\y) {$\scriptstyle{P, \bZ_4}$};
        \node at (1.8*\x,4.5*\y) {$\scriptstyle{H, \bZ_5}$};
        \node at (1.2*\x,2.1*\y) {$\scriptstyle{O, \bZ_5^2 \ltimes \bZ_4}$};

        \begin{scope}[xshift=3.5*\x cm]
          \draw[fill=white] (0*\x,1.7*\y) -- (2*\x,1.7*\y) -- (3*\x,3.3*\y) -- (1*\x,3.3*\y) -- cycle; 
          \draw[thick] (0.5*\x,2.5*\y) -- (2.5*\x,2.5*\y); 
          \fill (1.5*\x,2.5*\y) circle (.05); 
          \path (1.5*\x,-1*\y) node {$\scriptstyle{(Q, (0,0,0,2))}$};
          \node at (1.5*\x,3.6*\y) {$\scriptstyle{Q, \bZ_2}$};
          \node at (2.7*\x,2.1*\y) {$\scriptstyle{P, \bZ_4}$};
          \node at (1.2*\x,2.1*\y) {$\scriptstyle{O, \bZ_5 \ltimes \bZ_4}$};
        \end{scope}

        \begin{scope}[xshift=6.5*\x cm]
          \draw[thick] (0.5*\x,3*\y) -- (2.5*\x,3*\y); 
          \fill (1.5*\x,3*\y) circle (.05); 
          \draw[thick] (0.5*\x,2*\y) -- (2.5*\x,2*\y); 
          \fill (1.5*\x,2*\y) circle (.05); 
          \path (1.5*\x,-1*\y) node {$\scriptstyle{\bigsqcup (P, (0,0,0,h))}$};
          \node at (2.7*\x,1.6*\y) {$\scriptstyle{P, \bZ_4}$};
          \node at (1.2*\x,1.6*\y) {$\scriptstyle{O, \bZ_4}$};
        \end{scope}

        \begin{scope}[xshift=1*\x cm,yshift=-7*\y cm]
          \draw[thick] (0*\x,0.5*\y) -- (0*\x,4.5*\y); 
          \fill (0*\x,2.5*\y) circle (.05); 
          \draw[thick] (0.5*\x,0.5*\y) -- (0.5*\x,4.5*\y); 
          \fill (0.5*\x,2.5*\y) circle (.05); 
          \draw[thick] (1*\x,0.5*\y) -- (1*\x,4.5*\y); 
          \fill (1*\x,2.5*\y) circle (.05); 
          \draw[thick] (1.5*\x,0.5*\y) -- (1.5*\x,4.5*\y); 
          \fill (1.5*\x,2.5*\y) circle (.05); 
          \path (.75*\x,-1*\y) node {$\scriptstyle{\bigsqcup (H, (0,0,i,0))}$};
          \node at (-.35*\x,4.1*\y) {$\scriptstyle{H, \bZ_5}$};
          \node at (-.35*\x,2.4*\y) {$\scriptstyle{O, \bZ_5^2}$};
        \end{scope}

        \begin{scope}[xshift=4*\x cm, yshift=-7*\y cm]
          \fill (0*\x,1*\y) circle (.05); 
          \fill (0*\x,2*\y) circle (.05); 
          \fill (0*\x,3*\y) circle (.05); 
          \fill (0*\x,4*\y) circle (.05); 
          \draw[snake=brace] (-0.1*\x,0.9*\y) -- node[auto] {$\scriptstyle{4}$} (-0.1*\x,4.1*\y);
          \path (-0.2*\x,4.5*\y) node {$\scriptstyle{O, \bZ_5 \times D_{10}}$};

          \fill (1*\x,1*\y) circle (.05); 
          \fill (1*\x,2*\y) circle (.05); 
          \path (1*\x,3.2*\y) node {$\scriptscriptstyle{\vdots}$};
          \fill (1*\x,4*\y) circle (.05); 
          \draw[snake=brace] (0.9*\x,0.9*\y) -- node[auto] {$\scriptstyle{16}$} (0.9*\x,4.1*\y);
          \path (0.9*\x,4.5*\y) node {$\scriptstyle{O, \bZ_5^2}$};

          \fill (2*\x,1*\y) circle (.05); 
          \fill (2*\x,2*\y) circle (.05); 
          \path (2*\x,3.2*\y) node {$\scriptscriptstyle{\vdots}$};
          \fill (2*\x,4*\y) circle (.05); 
          \draw[snake=brace] (1.9*\x,0.9*\y) -- node[auto] {$\scriptstyle{20}$} (1.9*\x,4.1*\y);
          \path (1.9*\x,4.5*\y) node {$\scriptstyle{O, \bZ_5 \times \bZ_2}$};

          \fill (3*\x,1*\y) circle (.05); 
          \fill (3*\x,2*\y) circle (.05); 
          \fill (3*\x,3*\y) circle (.05); 
          \fill (3*\x,4*\y) circle (.05); 
          \draw[snake=brace] (2.9*\x,0.9*\y) -- node[auto] {$\scriptstyle{4}$} (2.9*\x,4.1*\y);
          \path (3*\x,4.5*\y) node {$\scriptstyle{O, \bZ_5}$};

          \fill (4*\x,1*\y) circle (.05); 
          \fill (4*\x,2*\y) circle (.05); 
          \path (4*\x,3.2*\y) node {$\scriptscriptstyle{\vdots}$};
          \fill (4*\x,4*\y) circle (.05); 
          \draw[snake=brace] (3.9*\x,0.9*\y) -- node[auto] {$\scriptstyle{48}$} (3.9*\x,4.1*\y);
          \path (4*\x,4.5*\y) node {$\scriptstyle{O, \bZ_4}$};

          \draw[snake=brace] (4.2*\x,0.5*\y) -- node[below] {$\scriptstyle{92}$} (-0.2*\x,0.5*\y);
          \path (2*\x,-1*\y) node {$\scriptstyle{\bigsqcup (O, (i_2,i_3,i_4,h))}$};
        \end{scope}        
      \end{tikzpicture}\]
      In particular we observe that the special point $O$ inside the component $(P, (0,0,0,h))$ is not really special, since its automorphisms group is the same as the one of $P$.
      
    \subsection{Torsion of order four} 
      We proceed in the same way as before, using Figure~\ref{fig:z4}. This time, all automorphisms live in the subgroup $H_4 \cong \bZ_4^3 \ltimes \bZ_2$ of $\bP\GL(8)$. The isomorphism associates to $(i_1, i_3, j_1, h)$ the matrix \[ \diag(\xi^{2 i_1}, 1, \xi^{2 i_3}, \xi^{i_3}, \xi^{i_1+i_3}, \xi^{i_1}, \xi^{j_1}, \xi^{-j_1}) P_{\sigma^h}\text{,}\] where $\sigma$ is the permutation $(1,3)(4,6)(7,8)$. Inside $H_4$, $G$ is generated by $(1,3,1,0)$.

      We obtain the following presentation in $H/G$:
      \begin{align*}
        \Aut(S_{\stamod_4}) &= \langle (0,0,0,0) \rangle\text{,}\\
        \Aut(S_{R_1}) &= \langle (2,2,0,0) \rangle\text{,}\\
        \Aut(S_W) &= \langle (0,0,0,1) \rangle\text{,}\\
        \Aut(S_{R_4}) &= \langle (2,2,0,0), (0,0,0,1) \rangle &&= \langle \Aut(S_W), \Aut(S_{R_1}) \rangle\text{,}\\
        \Aut(S_{R_{2,3}}) &= \langle (0,2,0,0) \rangle\text{,}\\
        \Aut(S_{S_4}) &= \langle (1,1,0,0) \rangle\text{,}\\
        \Aut(S_{S_5}) &= \langle (2,2,0,0), (0,2,0,0) \rangle &&= \langle \Aut(S_{R_1}), \Aut(S_{R_{2,3}}) \rangle\text{,}\\
        \Aut(S_{S_6}) &= \langle (1,1,0,0), (0,0,0,1) \rangle &&= \langle \Aut(S_W), \Aut(S_{S_4}) \rangle\text{,}\\
        \Aut(S_{S_7}) &= \langle (2,2,0,0), (0,2,0,0), (0,0,0,1) \rangle &&= \langle \Aut(S_W), \Aut(S_{S_5}) \rangle\text{,}\\
        \Aut(S_{S_{2,3}}) &= \langle (0,1,0,0) \rangle\text{,}\\
        \Aut(S_{T_{2,3}}) &= \langle (0,1,0,0), (2,2,0,0) \rangle &&= \langle \Aut(S_{S_4}), \Aut(S_{R_1}) \rangle\text{.}
      \end{align*}
      
      Now we can write the inertia stack:
      \begin{multline*}
        \I(\stamod_4) = (\stamod_4, (0,0,0,0)) \sqcup (R_1, (2,2,0,0)) \sqcup (W, (0,0,0,1)) \sqcup\\
        \sqcup (R_{2,3}, (0,2,0,0)) \sqcup (R_4, (2,2,0,1)) \sqcup (S_5, (0,2,2,0)) \sqcup\\
        \sqcup \bigsqcup_{\mathclap{i \in \{1,3\}}} (S_4, (i,i,0,0)) \sqcup \bigsqcup_{\mathclap{i \in \{1,3\}}} (S_{2,3}, (0,i,0,0)) \sqcup \bigsqcup_{\mathclap{i \in \{1,3\}}} (S_6, (i,i,0,1)) \sqcup\\
        \sqcup \bigsqcup_{\mathclap{i \in \{0,2\}}} (S_7, (0,2,i,1)) \sqcup \bigsqcup_{\mathclap{i \in \{1,3\}}} (T_{2,3}, (2,i,0,0))\text{.}
      \end{multline*}

      We do not try to draw the components, since there are many more than in the case $\nu = 5$ and more scattered through the various dimensions. We still have to show what are the automorphisms groups of the subcomponents of the components of the inertia stack. Again, these are trivially the original automorphisms groups if this is abelian; so the only case to study is $S_7$. Table~\ref{table:s7 automorphisms} sums up the results gathered with a \name{GAP} program similar to the one previously used.
      
      \begin{table}
        \begin{tabular}{ll}
          \toprule
          Component & $\Aut(S_7)$\\
          \midrule[\heavyrulewidth]
          $(\stamod_4, (0,0,0,0))$ & $D_8$\\
          $(R_1, (2,2,0,0))$ & $D_8$\\
          $(W, (0,0,0,1))$ & $\bZ_2^2$\\
          $(R_{2,3}, (0,2,0,0))$ & $\bZ_2^2$\\
          $(R_4, (2,2,0,1))$ & $\bZ_2^2$\\
          $(S_5, (0,2,2,0))$ & $\bZ_2^2$\\
          $(S_7, (0,2,2,1))$ & $\bZ_4$\\
          $(S_7, (0,2,0,1))$ & $\bZ_4$\\
          \bottomrule
        \end{tabular}
        \caption{Automorphisms groups for the $S_7$ subcomponents.}\label{table:s7 automorphisms}
      \end{table}

  \bibliographystyle{amsalpha}
  \bibliography{Godeaux1}
\end{document}